\documentclass[11pt, letterpaper, oneside]
{amsart}

\usepackage{latexsym, amsmath, amssymb, amsfonts, amscd,bm}
\usepackage{amsthm}
\usepackage{t1enc}
\usepackage[mathscr]{eucal}
\usepackage{indentfirst}
\usepackage{graphicx, pb-diagram}
\usepackage{fancyhdr}
\usepackage{fancybox}
\usepackage{enumerate}
\usepackage[all]{xy} 
\usepackage{hyperref}
\usepackage{color}

\usepackage{url}
\theoremstyle{plain}
\usepackage[all]{xy}
\newtheorem{thm}{Theorem}[section]

\newtheorem{prop}[thm]{Proposition}
\newtheorem{lemma}[thm]{Lemma}
\newtheorem{cor}[thm]{Corollary}

\theoremstyle{definition}
\newtheorem{defi}[thm]{Definition}

\theoremstyle{remark}
\newtheorem{remark}[thm]{Remark}
\newtheorem{ep}[thm]{Example}

\newcommand{\ZZ}{\ensuremath{\mathbb Z}}

\newcommand{\RR}{\ensuremath{\mathbb R}}
\newcommand{\g}{\ensuremath{\mathfrak{g}}}
\newcommand{\h}{\ensuremath{\mathfrak{h}}}

\newcommand{\bt}{\mathbf{t}}                  
\newcommand{\bs}{\mathbf{s}}                  

\newcommand{\cN}{\mathcal{N}}
\newcommand{\cL}{\mathcal{L}}
\newcommand{\cD}{\mathcal{D}}

\newcommand{\cI}{\mathcal{I}}
\newcommand{\cM}{\mathcal{M}}
\newcommand{\cC}{\mathcal{C}}
\newcommand{\cA}{\mathcal{A}}

\newcommand{\cS}{\mathcal{S}}
\newcommand{\cE}{\mathcal{E}}

\newcommand{\cG}{\mathcal{G}}

\newcommand{\un}{\underline}
\newcommand{\tGamma}{\tilde{\Gamma}}
\newcommand{\pd}[1]{\frac{\partial}{\partial #1}} 
\newcommand{\bbC}{\bar{\bar{\cC}}}
\newcommand{\bbS}{\bar{\bar{\cS}}}
 
\newcommand{\ndash}{\nobreakdash-\hspace{0pt}}

\begin{document}

\title{A supergeometric approach to Poisson reduction}
\author{A.S. Cattaneo}
\address{Institut f\"ur Mathematik, Universit\"at Z\"urich-Irchel, Winterthurerstr. 190, CH-8057 Z\"urich, Switzerland}
\email{alberto.cattaneo@math.uzh.ch}
\author{M. Zambon}
\address{Universidade do Porto, Departamentos de Matematica Pura, Rua do Campo Alegre 687, 4169-007 Porto, Portugal} \email{mzambon@fc.up.pt, marco.zambon@uam.es}

\begin{abstract}
This work introduces
a unified approach to the reduction of Poisson manifolds using their description by graded symplectic manifolds. This yields a
generalization of the classical Poisson reduction by distributions and allows one to  construct
actions of strict Lie 2-groups and to describe the corresponding reductions.
\end{abstract}
\maketitle

\setcounter{tocdepth}{1} 
\tableofcontents

\section{Introduction}\label{intro}
 
Many geometric structures (Poisson, Courant, generalized complex, 
\dots) may equivalently
be described in terms of graded symplectic manifolds endowed with functions satisfying structural
equations expressed in terms of the Poisson bracket. (Recall that a graded manifold is a supermanifold with a refined $\ZZ$\ndash grading,
i.e., a $\ZZ$\ndash grading whose reduction is the supermanifold $\ZZ_2$\ndash grading).

Reduction of these structures may then be understood as graded symplectic reduction compatible with the functions
and the structural equations. The advantage of this viewpoint is that in the (graded) symplectic world there is
only one reduction: namely, that of (graded) presymplectic submanifolds. Recall that a submanifold is presymplectic
if the restriction of the symplectic form has constant rank, and the reduction is by its kernel, also called
the characteristic distribution, which is always involutive. Functions whose restriction is invariant under this distribution descend to the quotient.
If the structural equations hold on the quotient, the geometric structure has been successfully reduced. This unified point of view
clarifies the various (often ad hoc) known reduction procedures and introduces new ones. Observe that (graded) presymplectic (or coisotropic)
submanifolds may also be induced by symplectic actions of (graded) Lie groups. This also defines suitable extensions of group actions
on the manifold carrying the geometric structure.

One step of our work consists of translating facts about graded (sub)manifolds into the usual language of differential geometry.
The second step consists of understanding (and translating) the conditions under which  the structural equations descend.
Observe that a sufficient condition is that the Hamiltonian vector fields of the functions defining the geometric structure
be tangent to the submanifold.
This very strong condition (which implies invariance and is equivalent to it in the coisotropic case) is far from being necessary.
However, various weaker sufficient conditions may be worked out.

In 
\cite{BCMZ}, which provided inspiration for the present paper,  this method is  applied  to the reduction of Courant algebroids and of generalized complex structures, 
recovering and extending the various known results. 
In this paper we concentrate on Poisson manifolds. In \textbf{Part 1} we recover the usual reduction of coisotropic and pre-Poisson submanifolds (\cite{CF, CZ}, see also \cite{MR}), 
the Marsden--Ratiu reduction \cite{MR} and various generalizations thereof which also go beyond the ones discussed in \cite{FZ}.
As an application, we obtain the somewhat surprising result that every Poisson manifold may be obtained by generalized reduction
from its cotangent bundle with canonical symplectic structure, see Section~\ref{sec:exdistr}.
In \textbf{Part 2} of the paper we deduce a generalized notion of compatible  actions.
The objects that act infinitesimally are certain DGLAs which correspond to crossed modules of Lie algebras (see Appendix \ref{A}). The corresponding global actions are by Lie 2-groups. Further, the global actions are Lie group actions in the category of groupoids rather than in the category of smooth manifolds, i.e. they are a categorification -- in the sense of Baez and Dolan \cite{BaezDolan} -- of the usual notion of Lie group action.
When the Lie 2-groups that act are just  ordinary Lie groups, several of our statements about actions and reductions specialize to some of the results of  \cite{FOR}.

We use graded geometry as a unifying guide to obtain the results of this paper. However, the results themselves may eventually be expressed using standard
differential geometry. The reader who is not interested in their derivation may see the classical statements directly in Prop.~\ref{prescase}
and in Thms.~\ref{A1} and~\ref{A2}  as well as
Prop.~\ref{glocM},
Prop.~\ref{groidgl},  Prop.~\ref{gloGamma} and Thm.~\ref{catgroupoact} (using Cor.~\ref{summ} to translate the assumptions into classical data).\\

\noindent\textbf{Plan of the paper.} In \textbf{Part 1} of this paper we obtain a statement about the reduction of Poisson manifolds $(M,\pi)$, where the input data is a submanifold of $M$ endowed with a suitable ``distribution''. In graded terms the Poisson manifold $M$ corresponds to a degree $1$ symplectic manifold $\cM$, and the input data to certain submanifolds of $\cM$. In Section \ref{sec:coiso} we consider coisotropic submanifolds of $\cM$ and obtain Prop.~\ref{coisocase}. To obtain a more general reduction statement, in Section \ref{sec:pres} we are forced to consider presymplectic submanifolds of $\cM$. Studying their geometry we obtain 
 Prop.~\ref{prescase}, a reduction statement for Poisson manifolds which  essentially amounts to the well-known Marsden--Ratiu theorem \cite{MR}. In Section \ref{sec:stages} we perform reduction in stages of presymplectic submanifolds of $\cM$ and obtain Thm.~\ref{A2},
which generalizes and improves  the Marsden--Ratiu theorem and Falceto's results with the second author \cite{FZ}
because it requires weaker assumptions. Finally in Section \ref{sec:exdistr} we present several examples in which the Poisson manifold $M$ is a cotangent bundle or the dual of a Drinfeld double.

In \textbf{Part 2} we consider  actions on the degree $1$  symplectic manifold $\cM$, which in turn induce actions on the symplectic groupoid $\Gamma$ of the Poisson manifold $(M,\pi)$.
In Sections \ref{actions}--\ref{sec:hamred} we consider an infinitesimal action on $\cM$ (which turns out to correspond to nice classical data), the corresponding global action,
 construct the global and
Marsden--Weinstein quotients of $\Gamma$, and translate them into classical terms.
In Section \ref{sec:groupGamma} we observe that an action on $\cM$ induces  an action on the symplectic groupoid $\Gamma$
in an interesting fashion. In Sections \ref{sec:glq} and \ref{Sec:MWGamma}
we construct its global and Marsden--Weinstein quotients. 
In Section \ref{sec:cga} we show that,  interestingly, the object acting on $\Gamma$ is not just a Lie group but rather a Lie 2-group, and that the action is  an action
in the category of Lie groupoids. In the forthcoming work  \cite{ZZL} by Zhu and the second author, the construction of the Lie 2-group action (Thm. \ref{catgroupoact}) is given a more conceptual explanation and extended to actions on any integrable Lie algebroid.
Finally \ref{secex} contains examples of the actions considered and their reductions. The Appendix collects some known facts about crossed modules and 2-groups.\\

\noindent\textbf{Notation.}
$M$ always denotes a smooth manifold. If $C$ is a submanifold of $M$ we denote its conormal bundle by
$N^*C:=\{\xi \in T^*M|_C: \langle \xi, TC \rangle =0\}$. If $E$ is a subbundle of $TM$ over $C$,  we use the notation $\tilde{\Gamma}(E)$
to denote sections of the vector bundle $TM$ which on $C$ lie in $E$, and  $C^{\infty}_E(M)$ to denote the functions on $M$ whose differential, at points of $C$, annihilates $E$. The notation $\un{C}$   denotes a natural quotient of $C$.
 If $M$
 is endowed with a Poisson bivector $\pi$, we denote by  $\sharp: T^*M \rightarrow TM$
the contraction $\xi \mapsto \pi(\xi)$. 
 
As a general rule, we denote objects related to graded manifolds by script letters.
$\cM$ is always a symplectic graded manifold, whose algebra of functions of degree $i$ we denote by $C^{\infty}_i(\cM)$ and
whose Poisson bracket we denote by $\{\cdot,\cdot\}$. 
If $\cI$ is a homogeneous ideal of functions on $\cM$, we denote its Poisson normalizer by $\cN(\cI)$.\\

\noindent\textbf{Acknowledgments.}
We are grateful to Henrique Bursztyn, Rajan Mehta and Dimitry Roytenberg for useful discussions and to Florian Sch\"atz and James Stasheff for comments that helped improve this manuscript. 
A.S.C.\ thanks the Centre de Recerca Matematica (Barcelona)
for hospitality. M.Z. also thanks Jo\~ao Martins and Chenchang Zhu for several useful explanations.
This work has been partially supported by SNF Grant 200020-121640/1, by
the European Union through the FP6 Marie Curie RTN ENIGMA (contract
number MRTN-CT-2004-5652),  by the European Science Foundation
through the MISGAM program, and Grant SB2006-0141(Spanish MEC).
Further M.Z. was partially supported by the Centro de Matem\'atica da Universidade do Porto, financed by FCT through the programs POCTI and POSI, and by the FCT program Ciencia 2007.


\section{The graded geometric description of Poisson manifolds}

Before describing how graded geometry can be used to perform reduction of Poisson manifolds,
we need to recall some notions of graded geometry. { The general notion of graded manifold was introduced by Voronov \cite{MR1958834}, and the related notion of N-manifold by \v{S}evera \cite{SWLett}.}
We will be extremely brief, since   we will need only graded manifolds of a very special form, for which we will provide explicit descriptions. 
\cite[Section 2]{Dima} and 
\cite[Sections 2--4]{CZbilbao} for more details.

Let $E=\oplus_{i<0}E_i$ be a graded vector bundle over a manifold
$M$, graded by negative integers $i$. The corresponding
\emph{N-manifold} is specified by the manifold $M$ and the graded
commutative algebra $\Gamma(S^{\bullet} E^*)$, the \emph{functions} on
the N-manifold, where $S^{\bullet}E^*$ denotes the graded symmetric algebra of $E^*$.
 The \emph{degree} of the N-manifold is the negative of the smallest $i$
appearing in the above direct sum. In this note we  consider only
N-manifolds of degree 1. When we endow them with graded symplectic
forms   which induce degree $-1$ Poisson brackets on the
graded algebra of functions, we speak of \emph{symplectic
N-manifolds of degree $1$}.

Let $M$ be a smooth manifold. We denote by $T^*[1]M$ the graded vector bundle $E$ over $M$ with $E_{-1}=T^*M$ and $E_i=\{0\}$ for $i\le -2$.
The N-manifold $T^*[1]M$ is canonically a symplectic
N-manifold of degree $1$, and all symplectic N-manifolds of degree $1$ arise this
way  {\cite{MR1230027}}.
{(For more information on the geometry of odd symplectic manifolds, see \cite{MR1413510}\cite{MR1249600}.)
}
The  algebra of functions $C^{\infty}(T^*[1]M)$ is
given by the multivector fields on $M$, and  the  Poisson bracket
$\{\cdot,\cdot\}$ is the Schouten bracket. 
Recall that a  \emph{Poisson structure} on $M$ is a bivector field $\pi \in \Gamma(\wedge^2TM)$ such that $[\pi,\pi]=0$. (This condition is equivalent to the fact that the bracket $\{f,g\}_M:=\pi(df,dg)$ on $C^{\infty}(M)$ satisfies the Jacobi identity.) A Poisson structure on $M$
can be regarded as a degree $2$ function $\cS$ on $T^*[1]M$ which Poisson commutes with itself.
The Poisson bracket on $M$ is recovered by a derived bracket construction:
\begin{equation}
\{f,g\}_M=\{\{\cS,f\},g\}
\end{equation}

We record this {\cite[Section 3.1]{MR1427124}}:
\begin{prop}\label{dimitri} There is a one-to-one correspondence between:
\begin{itemize}
\item  Poisson manifolds
\item     symplectic
N-manifolds of degree $1$, endowed with a degree $2$ function $\cS$ satisfying $\{\cS,\cS\}=0$.
\end{itemize}
 \end{prop}
\vspace{5mm}
\part{Reduction of Poisson manifolds by distributions}

\section{The main idea}\label{mainidea}

Our aim is to give a procedure which, starting with a Poisson manifold $(M,\pi)$ and 
making certain choices, allows one to construct   a new Poisson manifold.
The   one-to-one correspondence given by Prop. \ref{dimitri} allows us to phrase the problem as follows:
starting from a pair $(\cM,\cS)$ consisting of a symplectic
N-manifold  of degree $1$  and a degree $2$ self-commuting function $\cS$ on it, making certain choices, construct another such pair. There is a straightforward approach to the latter problem. 

In geometric terms it reads:
\begin{itemize}
\item[(a)] Take a graded presymplectic submanifold $\cC$ of $\cM$ so that the quotient $\un{\cC}$ by its characteristic distribution is smooth. Then $\un{\cC}$ is necessarily a symplectic
N-manifold  of degree $1$, hence of the form $T^*[1]X$ for some manifold $X$.
\item[(b)] Assume that 
\begin{itemize}
\item[1)] $\cS|_{\cC}$ descend to a function $\un{\cS}$ on the quotient $\un{\cC}$
\item[2)] $\un{\cS}$ Poisson commutes with itself.
\end{itemize}
Then $\un{\cS}$ corresponds to a Poisson bivector field on $X$.
\end{itemize}

Of course, since N-manifolds are defined in terms of graded commutative algebras (the ``functions''), all of the above has to be carried out in algebraic terms.  
\begin{itemize}
\item[(a)] A graded presymplectic submanifold   of $\cM$ is defined by a  homogeneous multiplicative ideal $\cI$ of $C^{\infty}(\cM)$. The quotient of the submanifold is defined by the graded Poisson algebra
$\cN(\cI)/(\cN(\cI) \cap \cI)$. Here $\cN(\cI)$ is the Poisson-normalizer of $\cI$ in $C^{\infty}(\cM)$.
\item[(b)] The assumptions on $\cS$ amount to 
\begin{itemize}
\item [1)] $\cS \in \cN(\cI)+\cI$ 
\item [2)] Conditions on ${\cS}$ guaranteeing that $\un{\cS}$ Poisson commutes with itself.
\end{itemize}
\end{itemize}

In the next Sections we will make  the above algebraic procedure explicit and show that it really corresponds to the geometric approach we outlined.
The hardest part will be finding \emph{sufficient} conditions on ${\cS}$ guaranteeing that $\un{\cS}$ Poisson commutes with itself.

\section{Coisotropic submanifolds}\label{sec:coiso}

In this Section we perform the simplest kind of graded symplectic reduction, namely that of coisotropic submanifolds. This will serve as a warm-up, since in later Sections we will obtain strictly stronger results.

Let $M$ be a smooth manifold. A \emph{graded submanifold}   $\cC$ of $\cM:=T^*[1]M$ is given by a homogeneous graded
ideal $\cI$ in $C^{\infty}(\cM)$ satisfying a smoothness property (Def. 7 of \cite{CZbilbao}), so it is generated by 
\begin{equation}\label{Izero}
\cI_0=Z(C):=\{f\in C^{\infty}(M):f|_C=0\}
\end{equation}
and
\begin{equation}\label{Ione}
\cI_1=\tilde{\Gamma}(E):=\{X\in
{\Gamma}(TM):X|_C\subset E\} 
\end{equation}
 for some closed submanifold $C\subset M$ and some
vector subbundle $E\rightarrow C$ of $TM\rightarrow M$. In other words, $\cC=E^{\circ}[1]$,
where $E^{\circ}\subset T^*M|_C$ denotes the annihilator of $E$.
Here and in the sequel we use the notation $\tilde{\Gamma}(\bullet)$
to denote sections of the vector bundle $TM$ which restrict to
sections of the subbundle $\bullet$.

 Denote by
$\cN(\cI)$ the Poisson normalizer of $\cI$, i.e., the set of
functions $\phi \in C^{\infty}(\cM)$ satisfying $\{\phi, \cI\}\subset \cI$.
 We have
\begin{equation}\label{NIzero}
 \cN(\cI)_0=\{f\in C^{\infty}(M):df|_C\subset
 E^{\circ}\}=:C^{\infty}_E(M)
\end{equation}
and
\begin{equation}\label{NIone}
\cN(\cI)_1=\{X\in \tilde{\Gamma}(TC):[X,\tilde{\Gamma}(E)]\subset \tilde{\Gamma}(E)\}.
\end{equation}
{Following the standard definition of coisotropic submanifold in Poisson geometry we define:}
\begin{defi}
The submanifold $\cC$ of $\cM$ is \emph{coisotropic} if its vanishing ideal
$\cI$ is closed under the Poisson bracket.
\end{defi} 

Now we assume that $\cC$ is a coisotropic submanifold, and spell out the coisotropicity condition.
By degree reasons
$\{\cI_0,\cI_0\}$ always vanishes. If  $X\in
\cI_1=\tilde{\Gamma}(E)$  and $f\in \cI_0=Z(C)$ we have
$\{f,X\}=-X(f)$. So $\{\cI_0,\cI_1\}\subset \cI_0$ is equivalent to
$E\subset TC$. If  $X,Y\in \cI_1$ then $\{X,Y\}=[X,Y]$, so
$\{\cI_1,\cI_1\}\subset \cI_1$ is equivalent to the involutivity of
the distribution $E$ on $C$.
Since by construction $\cI$ is a Poisson ideal in the Poisson algebra $\cN(\cI)$, the Poisson bracket descends making $\cN(\cI)/\cI$ into  a graded Poisson algebra.
 In degree $0$ it consists of the $E$-invariant functions on $C$, so let us assume that the quotient $\un{C}$ of $C$ by the foliation integrating $E$ be a smooth manifold (so that the projection map is a submersion).
In degree $1$, $\cN(\cI)/\cI$ consists of vector fields on $C$ which are
projectable w.r.t. the projection $C \rightarrow \un{C}$, modulo
vector fields lying in the kernel $E$ of the projection. In other words
$(\cN(\cI)/\cI)_1$ is isomorphic to the space of vector fields on
$\un{C}$. We conclude\footnote{One actually needs to justify why the algebra $\cN(\cI)/\cI$ is generated by elements in degree $0$ and $1$; this is done in more generality in the proof of Prop. \ref{pres2}.}
 that $\cN(\cI)/\cI$ is the graded Poisson
algebra of functions on a graded symplectic manifold if{f} $\un{C}$ is smooth, and
in that case it is the graded
 Poisson algebra of functions on
$T^*[1]\un{C}$.

Now we introduce a new piece of data: a Poisson bivector field $\pi$ on $M$. In graded terms it correspond to a
degree 2 function $\cS$ on $\cM$ satisfying $\{\cS,\cS\}=0$, see Prop. 
\ref{dimitri}. The function $\cS$ induces a function $\un{\cS}$ on $T^*[1]\un{C}$ if{f} $\cS\in \cN(\cI)$. In that case, by the way we defined the bracket on $\cN(\cI)/\cI$, it is clear that $\un{\cS}$ commutes with itself. Hence we obtain a reduced Poisson structure on $\un{C}$.
We spell out what it means for $\cS$ to lie in $\cN(\cI)$. Since for any function $f$ on $M$ we have  $\{\cS,f\}=[\pi,f]=\sharp df$, the condition $\{\cS,\cI_0\}\subset \cI_1$
is equivalent to $\sharp N^*C\subset E$. Here $\sharp: T^*M \rightarrow TM$
denotes contraction with the bivector $\pi$ and $N^*C:=\{\xi \in T^*M|_C: \langle \xi, TC \rangle =0\}$. Notice that in particular $C$ is a coisotropic submanifold of $(M,\pi)$. Further, for any vector field $X$ on $M$,
$\{\cS,X\}=[\pi,X]=-\cL_X \pi$, so
  $\{\cS,\cI_1\}\subset \cI_2$ means $(\cL_X \pi)|_C \in \Gamma(E\wedge TM|_C)$ for any $X\in \tilde{\Gamma}(E)$, which using eq. \eqref{Xfg} below is equivalent to $C^{\infty}_E(M)$ being closed under the Poisson bracket of $M$.

We summarize:
\begin{prop}\label{coisocase}Let $M$ be a smooth manifold.
A coisotropic submanifold $\cC$ of $T^*[1]M$ corresponds to a
submanifold $C$ of $M$ endowed with an integrable distribution $E$.
The coisotropic quotient of $\cC$ is smooth if{f} $\un{C}=C/E$ is
smooth, and in that case the coisotropic quotient is canonically
symplectomorphic to $T^*[1]\un{C}$.

Let $M$ be endowed with a Poisson structure $\pi$.
 The corresponding function $\cS$ on $\cM$
descends to a degree 2 self-commuting function on $T^*[1]\un{C}$
(which therefore corresponds to a Poisson structure on $\un{C}$)
if{f} $\sharp N^*C\subset E$ and $C^{\infty}_E(M)$ is closed under
the Poisson bracket.
\end{prop}

The Poisson-reduction result obtained from the above lemma is quite trivial, and is   a special case of  the Marsden--Ratiu theorem \cite{MR}.
In order to obtain less trivial and more interesting results we have to allow $\cC$ to be not just a coisotropic submanifold, but actually a presymplectic
 submanifold of $T^*[1]M$.

\section{Presymplectic submanifolds}\label{sec:pres}

In this Section we first define and characterize presymplectic submanifolds. Then we carry out the construction outlined in Section \ref{mainidea}: Step (a) in Prop. \ref{pres2}, Step (b1) in Lemma \ref{desc} and Step (b2) in Lemma \ref{halfnorm}. We summarize the results in Prop. \eqref{prescase}.


{Let $M$ be a smooth manifold and $\cC$ a submanifold of $T^*[1]M$.
Consider
$$Char(\cC):=\{X_{F}|_{\cC}: F\in \cN(\cI)\cap \cI\},$$
where $X_F$ denotes the hamiltonian vector field of $F$, and  $\cN(\cI)$ denotes the Poisson normalizer of the vanishing ideal $\cI$ of $\cC$. }
\begin{defi}\label{def:pres}
{A submanifold $\cC$ of $T^*[1]M$ is \emph{presymplectic} if{f}  $Char(\cC)$ is a (constant rank) distribution. In that case, 
$Char(\cC)$ is called
\emph{characteristic distribution}.}
\end{defi}
{Notice that if we apply the above definition to a submanifold $Y$ of an ordinary symplectic manifold $(X,\omega)$ we obtain the condition that 
the kernel of the pullback of $\omega$ to $Y$  has constant rank, recovering the usual notion of presymplectic submanifold in    symplectic geometry.} {For the notion of distribution see, e.g.,   \cite[Ch. 4]{var}}.

In this note we use an {equivalent characterization of presymplectic submanifolds,} which is more suitable for computations. {It relies on the notion of constant rank for matrices with entries in $C^{\infty}(\cC)$.}



\begin{defi}\label{presdef} Let $M$ be a smooth manifold. A submanifold $\cC$ of $T^*[1]M$ is \emph{presymplectic} if{f} 
locally there exist   homogeneous generators   $\phi_I$ of
$\cI$  for which the matrix $\;\;\{\phi_I,\phi_J\} \text{ mod }\cI\;\;$ has constant rank.
\end{defi}

\begin{ep}\label{exconst}
  a) Take $M=\RR^3$ with standard coordinates $x_i$, and denote the corresponding fiber coordinates on $T^*[1]\RR^3$ by $\theta_i$. Consider the submanifold of $T^*[1]\RR^3$ whose vanishing ideal $\cI\subset C^{\infty}(\cM)$ is generated by
$\phi_1=\theta_1$ and $\phi_2=\theta_2-x_1\theta_3$. The matrix of Poisson brackets is 
$\left(\begin{smallmatrix}  0& -\theta_3 \\
\theta_3 &0\\ \end{smallmatrix}\right).$ 
So the submanifold determined by $\cI$ is not presymplectic.

b) Now  consider the larger ideal obtained adding the generator $\phi_0=x_2$. We have $$\{\phi_I,\phi_J\}=
\begin{pmatrix}  0& 0 & -1\\
0& 0 & -\theta_3 \\
1& \theta_3 &0 \end{pmatrix}. 
$$
Hence the submanifold of $T^*[1]M$ determined by $\phi_0,\phi_1,\phi_2$ is presymplectic. 
\end{ep}

In the previous Section we saw that (graded) submanifolds of $T^*[1]M$ are of the form $E^{\circ}[1]$ for some  
vector subbundle $E\rightarrow C$ of $TM\rightarrow M$. We now characterize the presymplectic condition in terms of $E$ and $C$.  

\begin{prop}\label{pres1}
  $\cC=E^{\circ}[1]$ is a graded presymplectic submanifold if{f} $F:=TC\cap E$ is a constant rank, involutive distribution on $C$.
\end{prop}

\begin{proof}
Locally on $C$ pick generators $f_i$ and $X_{\alpha}$ of $\cI$, of degree $0$ and $1$ respectively. Denote them collectively by $\phi_I$.
The submanifold $\cC$ is presymplectic if{f} the matrix $\{\phi_I,\phi_J\} \text{ mod }\cI$   has constant rank.

Assume first that $\cC$ is presymplectic. 
The constant rank of $\{\phi_I,\phi_J\} \text{ mod }\cI$ implies that the degree $0$ matrix $\{f_i,X_{\alpha}\}$ has constant rank, so
$E/(TC\cap E)$ and therefore $TC\cap E$ have constant rank.
Now refine further the choice of constraints as follows: $f_i\in Z(C)$ so that the last elements annihilate $E$,
$X_{\alpha}\in \tGamma(E)$ so that the
first elements lie in $TC\cap E$. We write down the 4 by 4 block-matrix 
$$
\{\phi_I,\phi_J\} \text{ mod }\cI=
\left(\begin{array}{c c|c c}
 0 &0&0&{\sigma} \\
0 &0&0&0\\
\hline
0 &0&\gamma&*\\
{\tau} &0& \delta & * \\
\end{array}\right),
$$
for which ${\sigma},{\tau}$
are \emph{invertible}  matrices of degree zero. The columns  of the block $\delta$ can be expressed as $C^{\infty}_1(\cC)$-linear combinations of the columns   of ${\tau}$.
 Since $\gamma$ consists of degree $1$ elements, it
vanishes on $C\subset \cC$, so we conclude that  $\{\phi_I,\phi_J\} \text{ mod }\cI$ has constant rank   if{f} the entries of $\gamma$ lie in $\cI$. This is equivalent to
  $[X_{\alpha},X_{\beta}]\subset \cI_1=\tGamma(E)$
whenever  $X_{\alpha},X_{\beta} \in \tGamma(TC\cap E)$, which in turn is equivalent to $TC\cap E$ being involutive.

Now assume that  $TC\cap E$ is of constant rank and involutive. Choosing the constraints $\phi_I$ as above and reversing the above argument  we conclude that
$\{\phi_I,\phi_J\} \text{ mod }\cI$  has constant rank.
\end{proof}

\begin{ep}\label{contact}
The graded submanifold of $T^*[1]\RR^3$ considered in Ex. \ref{exconst} a) is not presymplectic. It is given by $E^{\circ}[1]$ where $E$ is the kernel of the standard contact form $x_1dx_2+dx_3$ on $\RR^3$, so in particular $E$ is not involutive. This is consistent with Prop. \ref{pres1}.
\end{ep}

\begin{remark}
{Prop. \ref{pres1} states in particular that  $\cC=E^{\circ}[1]$ being
presymplectic (a constant rank condition) implies the involutivity of  $TC\cap E$. 
This is consistent with the  combination of the following two facts: first, if $\cC$ is presymplectic, then $TC\cap E$ agrees with the restriction of the degree zero component of $Char(\cC)$ to $C$ (this follows immediately from Prop. \ref{pres1} together with the later Lemma \ref{ext} and eq. \eqref{INI1}). Second,
 $Char(\cC)$ is involutive, being the kernel of a constant rank  \emph{closed} 2-form. 
}

{Regarding the converse implication in Prop. \ref{pres1},
notice that given a vector subbundle $E\to C$  such that $TC\cap E$ has constant rank, $TC\cap E$ is not necessarily involutive, hence 
 $E^{\circ}[1]$ is not necessarily presymplectic. A counterexample is given in Ex. \ref{contact}.}
 \end{remark}

\begin{defi} 
Let $\cC$ be a presymplectic submanifold of $T^*[1]M$. 
The \emph{presymplectic quotient of $\cC$} is
{the quotient of $\cC$ by the characteristic distribution $Char(\cC)$.}
\end{defi}
\begin{remark}\label{rem:charN}
{The set   of  functions on $\cC$ which are invariant under the characteristic distribution satisfies $$C^{\infty}(\cC)^{Char(\cC)}=\cN(\cI)/(\cN(\cI) \cap\cI),$$ as we will show in eq. \eqref{justy} in the proof of Prop. \ref{pres2}. The latter has an induced graded Poisson algebra structure,
computed by lifting to functions in $\cN(\cI)$ and applying the bracket $\{\cdot,\cdot\}$ of $C^{\infty}(T^*[1]M)$.
Hence when the presymplectic quotient of $\cC$ is smooth, it will be a graded Poisson manifold (indeed symplectic, see Prop. \ref{pres2}).}
 \end{remark}

Before determining the {presymplectic} quotient of $\cC$ we need two technical lemmas.
\begin{lemma}\label{ext}
Let $M$ be a manifold, $C$ a submanifold, and $E\subset TM|_C$  a subbundle so that $F:=TC \cap E$ is an involutive constant rank distribution on $C$ so that $\un{C}:=C/F$ is smooth. Then every vector field on $C$ which is projectable w.r.t. $C\rightarrow \un{C}$ can be extended to a vector field on $M$ lying in 
$\{X\in \tilde{\Gamma}(TC):[X,\tilde{\Gamma}(E)]\subset \tilde{\Gamma}(E)\}.$
\end{lemma}  
\begin{proof}
 Fix a subbundle $B$ which is a  complement to $TC\cap E$ in $E$. Extend it to a complement $\nu C$ of $TC$ in $TM|_C$, i.e., to a choice of normal bundle for $C$. 
Choose an Ehresmann  connection for the vector bundle $\nu C \rightarrow C$ which restricts to a connection on the subbundle $B \rightarrow C$. We claim that if $X$ is a projectable vector field on $C$, then its horizontal lift $X^H$, a vector field on $\nu C \cong M$, is an extension with the  required property. Here we fix an identification of the total space of $\nu C$ with (a tubular neighborhood of $C$ in) $M$.

Let $Y_F \in \tilde{\Gamma}(F)$, i.e., $Y_F$ is a vector field on $M\cong \nu C$ whose restriction to $C$ lies in $F$. Then $[X^H,Y_F]|_C=[X,(Y_F)|_C]\subset F$, since by assumption $X$ is a projectable vector field. Further take a section of the vector bundle $B \rightarrow C$, extend it by translation to a vertical vector field on $B$, and then to a vector field $Y_B$ on $\nu C\cong M$. Since the flow of $X^H$ preserves the fibers of 
$B \rightarrow C$ it follows that $[X^H, Y_B]|_B$ is a vertical vector field on $B$.
Altogether this shows $[X^H, Y_F+Y_B]\in \tilde{\Gamma}(E)$. Since any element of $\tilde{\Gamma}(E)$
can be written as   $Y_F+Y_B$, up to a vector field vanishing on $C$, we are done. 
\end{proof}
 
\begin{lemma}\label{Redu}
Let $\cN$ be an N-manifold, and denote by $N$ its body. Let $\cD$ be an involutive distribution 
 on $\cN$. Assume that  $\un{N}$, the quotient of  $N$ by the (degree zero part of) $\cD$, is a smooth manifold such that the projection  $pr \colon N \rightarrow \un{N}$ is a submersion. Suppose that the following technical conditions are satisfied for every open subset $U \subset N$:
\begin{itemize}
 \item[i)]  For every  $\cD$-invariant $f\in C^{\infty}(\cN)_U$ of degree $\le deg(\cN)$ there exists a $\cD$-invariant ${F}\in C^{\infty}(\cN)_{pr^{-1}(pr(U))}$ with ${F}|_U=f$
\item[ii)] If two $\cD$-invariant functions ${F},{G}\in C^{\infty}(\cN)_{pr^{-1}(pr(U))}$ agree on $U$, then they are equal.
\end{itemize}
Then  $$C^{\infty}(\un{\cN})_V:=C^{\infty}(\cN)_{pr^{-1}(V)}^{\cD-\text{invariant}}\;\;\;\;\;\ \text{for all open subsets }V\subset \un{N}$$
defines a sheaf (over $\un{N}$) of graded commutative algebras \emph{generated by their elements in degrees $0,\dots,\text{deg}(\cN)$}.
\end{lemma}
\begin{proof}
By the 
Frobenius theorem ({see, e.g.,   \cite[Ch. 4]{var}})
there exist graded local coordinates on $\cN$ adapted to the distribution $\cD$. Clearly the invariant functions on $\cN$ are  those that depend only on the coordinates transverse to $\cD$,
 so it is clear that  on \emph{small} open sets of $N$ the algebra of invariant functions is generated by (invariant)  elements in degrees $0,\dots,\text{deg}(\cN)$.

We show that this is true for open sets of the form $pr^{-1}(V)$ where $V$ is a  small open subset of $\un{N}$.
Indeed let ${F}$ be an invariant function in  $C^{\infty}(\cN)_{pr^{-1}(V)}$. Take a small enough open set $U \subset N$ with $pr(U)=V$. By the first paragraph 
  $F|_{U}\in C^{\infty}(\cN)_{U}$  is a sum of products of  invariant element of $C^{\infty}(\cN)_{U}$ of degrees $0,\dots,\text{deg}(\cN)$.
By assumption $i)$ we can extend them to 
invariant elements of $C^{\infty}(\cN)_{pr^{-1}(V)}$,  which combine into an invariant function $\hat{F}\in C^{\infty}(\cN)_{pr^{-1}(V)}$  extending $F|_{U}$, which in turn by assumption $ii)$ must be equal to $F$. We conclude that  the invariant functions in $C^{\infty}(\cN)_{pr^{-1}(V)}$ form a graded commutative algebra generated by its elements in degrees $0,\dots,\text{deg}(\cN)$.
For arbitrary open sets of $\un{N}$ the same holds using a partition of unity argument on $\un{N}$.

The fact that $C^{\infty}(\un{\cN})$ is a sheaf follows immediately from  $C^{\infty}({\cN})$ begin a sheaf and 
$\cD$-invariance being a local property.
\end{proof}

\begin{prop}\label{pres2} 
 The presymplectic quotient of $\cC=E^{\circ}[1]$
 is smooth
if{f} the quotient $\un{C}:=C/(TC\cap E)$ is smooth, and in this case
it is canonically isomorphic to $T^*[1]\un{C}$ as a graded symplectic manifold.
\end{prop}

\begin{proof}
To start with, let us look at the degree $0$ and $1$ functions on $\cM$ which restrict to $Char(\cC)$-invariant functions on $\cC$.
To do so we make  use of
\begin{eqnarray}\label{INI}
(\cN(\cI)\cap \cI)_0&=&
\{f\in Z(C):df|_C\subset E^{\circ}\},\\
\label{INI1}
(\cN(\cI)\cap \cI)_1&=&\{X\in \tilde{\Gamma}(TC\cap
E):[X,\tilde{\Gamma}(E)]\subset \tilde{\Gamma}(E)\}.
\end{eqnarray}
Let $f\in C^{\infty}_0(\cM)$.
 The condition $\{f, \cN(\cI)\cap \cI\}\subset  \cI$ means that $f|_C$ is a function on $C$ that is constant along $TC\cap E$. Let $X\in C^{\infty}_1(\cM)$.
The condition $\{X, (\cN(\cI)\cap \cI)_0\}\subset  \cI$ means that $X  \in \tilde{\Gamma}(E+TC)$, so we may assume that $X\in \tilde{\Gamma}(TC)$ by adding an element of $\tilde{\Gamma}(E)=\cI_1$. The condition 
$\{X, (\cN(\cI)\cap \cI)_1\}\subset  \cI$ then says that $X|_C$ is a basic vector field with respect to the projection $pr \colon C \rightarrow \un{C}$.

We claim that $C^{\infty}(\cC)^{Char(\cC)}$ is generated by elements in degrees $0$ and $1$.
The characteristic distribution is involutive since $\cN(\cI)\cap \cI$ is closed under the Poisson bracket. Further the invariant functions in degrees $0$ and $1$, by the above description, are given
by functions and vector fields on $U$ which are projectable with respect to $pr|_U$.
They satisfy the assumptions of Lemma \ref{Redu}:
assumption $i)$ by Lemma \ref{ext}, and assumption $ii)$ because vector fields on $C$ which projects to the same vector field on $\un{C}$ must differ by sections of $E\cap TC$.
  Hence the claim follows immediately from Lemma \ref{Redu}.

Next we claim that\footnote{To prove this claim it is important that $\cI$ is not just any homogeneous ideal on $C^{\infty}(\cM)$, but one that defines a (presymplectic) submanifold.}
\begin{equation}\label{eq:INI}
\{F\in C^{\infty}(\cM): \{F,\cN(\cI)\cap \cI\}\subset  \cI\}=\cN(\cI)+\cI.
\end{equation}
 We only show the inclusion ``$\subset$'', because the other one is clear. Elements of the L.H.S. are exactly the functions on $\cM$ whose restriction to $\cC$ lies in $C^{\infty}(\cC)^{Char(\cC)}$, and we just saw that the latter is generated by elements in degrees $0$ and $1$. Hence it is sufficient to show the inclusion ``$\subset$'' only for elements of the L.H.S. of degrees $0$ and $1$. We saw above that such elements in degree $0$ consist of $f\in C^{\infty}_0(\cM)$ such that $f|_C$ is constant along $TC\cap E$, hence they can be extended to a function in $C_E^{\infty}(M)=\cN(\cI)_0$. We also saw that in degree $1$, up to elements of $\cI_1$, they consist of  $X\in C^{\infty}_1(\cM)$ such that $X|_C$ is a basic vector field with respect to the projection $C \rightarrow C/(TC\cap E)$, so by Lemma \ref{ext} we can extend them to an element of $\cN(\cI)_1$.
 This proves the claim. 

Quotienting the identity \eqref{eq:INI} by $\cI$ we obtain
\begin{equation}\label{justy}
C^{\infty}(\cC)^{Char(\cC)}=(\cN(\cI)+\cI)/\cI=\cN(\cI)/\cN(\cI)\cap \cI.
\end{equation}

Now we show that $\cN(\cI)/(\cN(\cI)\cap \cI) \cong C^{\infty}(T^*[1]\un{C})$. Since both
sides  are generated by elements in degrees $0$ and $1$, it suffices to show that the elements in degrees $0$ and $1$ agree.
Using \eqref{INI} we see that $\cN(\cI)_0/\cN(\cI)_0\cap \cI_0$
consists of  the $E\cap TC$-invariant functions on $C$, which agree with the space of  functions on a smooth manifold if{f} $\un{C}$ is a smooth manifold. In this case $\cN(\cI)_0/\cN(\cI)_0\cap \cI_0$ is canonically identified with 
$C^{\infty}(\un{C})$. In degree $1$ we have a map
$$\cN(\cI)_1 \rightarrow
\frac{ \{Y\in  {\Gamma}(TC):[Y,{\Gamma}(E\cap TC)]\subset
{\Gamma}(E\cap TC)\} }{ {\Gamma}(E\cap TC)}\cong\{\text{vector fields on }\un{C}\}$$ obtained by
restricting to $C$. The kernel is $\cN(\cI)_1\cap \cI_1 $,
and using Lemma \ref{ext}  we see that this map is also surjective. Hence
$\cN(\cI)_1/\cN(\cI)_1\cap \cI_1$ is canonically isomorphic to the space of  vector fields
on $\un{C}$. We conclude that the projection $pr \colon C \rightarrow \un{C}$ induces an isomorphism 
$C^{\infty}(\cC)^{Char(\cC)}\cong C^{\infty}(T^*[1]\un{C})$, which furthermore preserves brackets because $pr$ preserves the Schouten bracket.
\end{proof}

\begin{remark}\label{remC}
We describe in classical terms the construction of the quotient $\un{\cC}$ from $\cC$, which in Prop.  \ref{pres2} has been described algebraically (i.e., in terms of functions).
Notice that the submanifold $C$ is endowed with the foliation integrating $E\cap TC$, hence  the normal bundle to the foliation, which is  $TC/(E\cap TC)$, is endowed with a flat $E\cap TC$-connection
(the Bott connection), defined using the Lie bracket. Hence  the dual bundle
is endowed with the dual connection, which is also flat. Explicitly, the
dual bundle is $(E\cap TC)^{\circ}/TC^{\circ}$ and the dual connection
$\nabla_X \un{\xi}:=\un{\cL_X\xi}$ where $X\in E\cap TC$ and $\xi\in \Gamma((E\cap TC)^{\circ})$
is a lift of $\un{\xi}\in \Gamma((E\cap TC)^{\circ}/TC^{\circ})$. (Here all annihilators are taken in $TM|_C$).

The construction of $\un{\cC}$ is as follows: starting from the vector bundle $E^{\circ}\rightarrow C$, quotient the fibers by the intersection  with $TC^{\circ}$,
to obtain $E^{\circ}/(TC^{\circ}\cap E^{\circ})\cong (E\cap TC)^{\circ}/TC^{\circ}\rightarrow C$, then identify fibers lying over the same leaf of $TC\cap E$ using the
flat connection $\nabla$. The parallel sections of $\nabla$ are exactly the pullbacks of 1-forms on $\un{C}$,
so the resulting quotient is $T^*\un{C}$.
\end{remark}

Now endow $M$ with a Poisson tensor $\pi$, corresponding to a function $\cS$ on $T^*[1]M$.
We address the issue of when the function $\cS$ induces a
function $\un{\cS}$ on the quotient $\un{\cC}:=T^*[1]\un{C}$, where
 $\un{C}:=C/F$ for 
$F:=TC\cap E$. {Recall that the functions on $\un{\cC}$ are realized as $\cN(\cI)/(\cN(\cI) \cap\cI)$, see
Remark \ref{rem:charN}.}

\begin{lemma}\label{desc}
$\cS$ descends if{f}
\begin{eqnarray}\label{condi1}
\sharp E^{\circ}&\subset& E+TC\\
\label{condi2}
\{C^{\infty}_E(M),C^{\infty}_E(M)\}_M&\subset& C^{\infty}_{F}(M).
\end{eqnarray}In this case the induced almost-Poisson bracket on $C^{\infty}(\un{C})$ is computed by lifting
to functions in $C^{\infty}_E(M)$ and applying the Poisson bracket of $M$.
 \end{lemma}
\begin{proof}
 $\cS$ descends if{f} its image under the restriction map $C^{\infty}(\cM)\rightarrow C^{\infty}(\cM)/\cI$
lies in
 $\cN(\cI)/\cN(\cI)\cap \cI$, i.e., if{f} $\cS$ lies in $\cN(\cI)+ \cI$. This is equivalent   to $\{\cS,\cN(\cI)\cap \cI\}\subset  \cI$ by eq. \eqref{eq:INI}.

Using eq. \eqref{INI} one sees that $\{\cS,\cN(\cI)_0\cap \cI_0\}\subset \cI_1$ is  equivalent to eq.  \eqref{condi1}.
 $\{\cS,\cN(\cI)_1\cap \cI_1\}\in \cI_2$ means $[\pi,X]
 \in \tilde{\Gamma}(E\wedge TM|_C)$ for all $X\in \cN(\cI)_1\cap \cI_1$.
This is seen to be equivalent to eq.
  \eqref{condi2} as follows: apply the identity
\begin{equation}\label{Xfg} X\{f,g\}=
(\cL_X
\pi)(df,dg)+\pi(d(Xf),dg)+\pi(df,d(Xg))\end{equation}
to $f,g \in C^{\infty}_E(M)$, and observe that the last two terms vanish on $C$ because of \eqref{condi1}.

To compute the Poisson bracket of functions $\un{f},\un{g}$ on $\un{C}$ we fix extensions
$\hat{f},\hat{g} \in C^{\infty}_E(M)$. Fix a choice of function
$\hat{\cS}\in \cN(\cI)$ with $\hat{\cS}-\cS\in \cI$ (therefore also $\un{\hat{\cS}}=\un{\cS}$). Such a function exists because the fact that $\cS$ descends means that $\cS\in \cN(\cI)+\cI$.   We compute
$$\{\un{f},\un{g}\}_{\un{C}}=
\{\{\un{\cS},\un{f}\}_{\un{\cC}},\un{g}\}_{\un{\cC}}
=\{\{\hat{\cS},\hat{f}\},\hat{g}\} \text{ mod } \cI =  \{\{\cS,\hat{f}\},\hat{g}\}\text{ mod } \cI= (\{\hat{f},\hat{g}\}_M)|_C,$$
where the third equality holds because $\hat{\cS}-\cS\in \cI$ and $\hat{f},\hat{g}\in \cN(\cI)$.
  \end{proof}

When $\cS$ descends,
 $\un{\cS}$ might not commute with itself, as in Ex. \ref{counterex}. The reason is that  the Poisson bracket on
 $\cN(\cI)/\cN(\cI)\cap \cI$ is computed taking the derived bracket not with $\cS \in \cN(\cI)+\cI$, but rather with an extension of $\cS|_{\cC}$ lying in $\cN(\cI)$.

It is clear that if $\cS$ lies in $\cN(\cI)$, then the induced function on $\un{\cC}$ still commutes with itself. It turns out that it suffices to require that $\cS$ satisfies the normalizer condition in degree $0$.
\begin{lemma}\label{halfnorm}
Suppose that $\cS$ descends.
A sufficient condition to guarantee
$\{\un{\cS}, \un{\cS}\}=0$ is $\{\cS,\cI_0\}\subset \cI_1$ (or equivalently
$\sharp TC^{\circ}\subset E$).
\end{lemma}
\begin{proof}
Let $\un{f},\un{g},\un{h}$ functions on $\un{C}$, and fix extensions $\hat{f},\hat{g},\hat{h}$
 to functions in $\cN(\cI)_0=C^{\infty}_E(M)$.
Fix a choice of function
$\hat{\cS}\in \cN(\cI)$ with $\hat{\cS}-\cS\in \cI$ (therefore also $\un{\hat{\cS}}=\un{\cS}$). 
 To simplify the notation we denote $\un{V}:=\{\un{f},\un{g}\}_{\un{C}}$, and 
$\hat{V}:=\{\{\hat{\cS},\hat{f}\},\hat{g}\}$ is an extension lying in $\cN(\cI)_0$.
We have
\begin{eqnarray*}
\{\{\un{f},\un{g}\}_{\un{C}},\un{h}\}_{\un{C}}=
\{\un{V},\un{h}\}_{\un{C}}&=& \{\{\un{\cS},\un{V}\}_{\un{\cC}},\un{h}\}_{\un{\cC}}\\
&=&\{\{\hat{\cS},\hat{V}\},\hat{h}\}\text{ mod } \cI\\ \label{jacobib}
&=& \{\{\cS,\hat{V}\},\hat{h}\}\text{ mod } \cI,
\end{eqnarray*}
where we used the property $\hat{\cS}-\cS\in \cI$ in the last equality. The same property
assures that
$$\hat{V}:=\{\{\hat{\cS},\hat{f}\},\hat{g}\}=\{\{{\cS},\hat{f}\},\hat{g}\}+k$$ for some $k\in \cI_0$.
Putting this together we obtain
$$\{\{\un{f},\un{g}\}_{\un{C}},\un{h}\}_{\un{C}}=\{\{\hat{f},\hat{g}\}_M, \hat{h}\}_M+ \{\{\cS,k\},\hat{h}\} \text{ mod } \cI.$$
By assumption $\{\cS,k\}\in \cI_1$, and since $\hat{h}\in \cN(\cI)$ the whole second term lies in $\cI$. Taking the cyclic sum shows that the Jacobiator of $\un{f},\un{g},\un{h}$
vanishes. 

Since $\{\bullet,\bullet\}_{M}$ satisfies the Jacobi identity we conclude that $\{\bullet,\bullet\}_{\un{C}}$ also satisfies the Jacobi identity, i.e., that $\{\un{\cS}, \un{\cS}\}=0$.
\end{proof}

We summarize in classical terms the results obtained in this Section:
\begin{prop}\label{prescase}
Let $C$ be a submanifold of a Poisson manifold $(M,\pi)$ and $E$ a subbundle of $TM|_C$ such
 $F:=TC\cap E$ is a constant rank, involutive distribution on $C$. Assume that
  the quotient $\un{C}:=C/(TC\cap E)$ is smooth. If
\begin{equation*} \sharp E^{\circ}\subset TC \end{equation*}
and 
\begin{equation*}
\{C^{\infty}_E(M),C^{\infty}_E(M)\}_M\subset C^{\infty}_{F}(M).
\end{equation*}
then $\un{C}$ has an induced Poisson structure. Its Poisson  bracket
  is computed lifting functions on $\un{C}$ to functions in
$C^{\infty}_E(M)$.
\end{prop}
\begin{proof}
By Prop. \ref{pres1} and  \ref{pres2} $\cC:=E^{\circ}[1]$ is a presymplectic submanifold of the symplectic graded manifold $T^*[1]M$ such that its quotient $\un{\cC}$ is canonically symplectomorphic to $T^*[1]\un{C}$. By Lemma \ref{desc} the degree 2 function $\cS$ on $T^*[1]M$, which encodes the Poisson bivector $\pi$, descends to a function $\un{S}$ on $T^*[1]\un{C}$, and by Lemma \ref{halfnorm} $\un{\cS}$ commutes with itself. Hence $\un{S}$ corresponds to a Poisson bivector on $\un{C}$. Its Poisson bracket is computed as in Lemma \ref{desc}.
\end{proof}
Notice that when $E\subset TC$ 
we recover exactly the reduction statement in Prop. \ref{coisocase}. Prop. \ref{prescase} is a very mild extension of the Marsden--Ratiu  theorem \cite{MR} (see \cite{FZ}, where Prop. \ref{prescase} above appears as Prop. 4.1).

\section{Reduction in stages}\label{sec:stages}

In this Section we derive a sufficient condition for $\{\un{\cS},\un{\cS}\}=0$ which is weaker than the one of Lemma \ref{halfnorm}. To do so,  we perform reduction in stages in an algebraic fashion.
The corresponding geometric picture is the following refinement of the one outlined in Section \ref{mainidea}.

\begin{itemize}
\item[(a)] We imbed $\cC$ in a larger
 coisotropic\footnote{It seems more natural to require $\cA$ to be presymplectic instead of coisotropic. However this  more general statement delivers   conditions which reduce to Prop. \ref{A1} below.} 
submanifold $\cA$ of $\cM$. We assume that
 the presymplectic quotient
$\un{\cC}$ is smooth. Locally (i.e., if we choose small open sets $U\subset M$ and replace
 $\cM$ by $\cM_U$, $\cC$ by $\cC_{C\cap U}$ and  $\cA$ by $\cA_{A\cap U}$) 
we perform the two-stage reduction
 \begin{itemize}
\item[--]  take the image $\bar{\cC}$ of $\cC$ under the projection
$\cA\rightarrow \bar{\cA}:=\cA/T\cA^{\omega}$; assuming that $T\cC
\cap T\cA^{\omega}$ has constant rank,  $\bar{\cC}$ is a
presymplectic submanifold. 
\item[--] take the presymplectic quotient $\bbC$ of 
$\bar{\cC}$. It is a  symplectic graded manifold  
  symplectomorphic\footnote{If instead of restricting ourselves to
small open subsets $U$ of $M$ we work globally, we just get a map $\bbC \rightarrow \un{\cC}$ preserving symplectic structures.}
 to $\un{\cC}$.
\end{itemize}

\item[(b)] Assume that
\begin{itemize}
\item[1)] $\cS$ descends to $\un{\cC}$
\item[2)]  locally, $\cS$  descends to a function $\bar{\cS}$ on
$\bar{\cA}$
\item[$2^{\prime}$)]  locally, $\bar{\cS}$ satisfies the condition of Lemma \ref{halfnorm}, i.e.,
$\{\bar{\cS},(\cI_{ \bar{\cC}})_0\}\subset (\cI_{ \bar{\cC}})_1$.
\end{itemize}
Since $\cA$ is coisotropic, by   condition $2)$
the function $\bar{\cS}$ on
$\bar{\cA}$  self-commutes, and
 Lemma \ref{halfnorm} together with condition $2^{\prime}$)
 imply that the function $\bbS$ on $
\bbC$  commutes with itself.
By (a) hence the function $\un{\cS}$ on $\un{\cC}$ also commutes with itself\footnote{Notice that it is not relevant here
whether $\un{\cC}$ is globally diffeomorphic to $\bbC$:  we assume that we have a well-defined
  function $\un{\cS}$ on $\un{\cC}$ and use $\bbC$
only to check   a local property of $\un{\cS}$, namely that it
commutes with itself.}, and therefore corresponds to a Poisson bivector field on the body of $\un{\cC}$.
\end{itemize}

In the remainder of this Section we will phrase  (a slightly more general version of) the above construction
in algebraic terms and proof that it really delivers a Poisson structure on the body of $\un{\cC}$. We perform Step  (a) in Lemma \ref{lift} and Step (b) in
Thm. \ref{coisopre}, and we translate into  classical geometrical terms in Thm. \ref{A1} and Thm. \ref{A2}.\\

Let $M$ be a smooth manifold, $\cC$ a presymplectic submanifold of $\cM=T^*[1]M$ and $\cA$ a coisotropic submanifold
containing $\cC$. Write $\cC=E^{\circ}[1]$ for a subbundle $E \rightarrow C$ of $TM$.
$E\cap TC$ is an involutive
constant rank distribution since $\cC$ is presymplectic  (Prop. \ref{pres1}), and we   assume that $\un{C}:=C/(E\cap TC)$ is
smooth.  Write
$\cA=D^{\circ}[1]$ for another   subbundle $D \rightarrow A$. $D$ an
 integrable distribution on $A$ since $\cA$ is coisotropic (Prop \ref{coisocase}). Further $C \subset A$ and
$D|_C\subset E$, since $\cC \subset \cA$.
\begin{lemma}\label{lemmaconstrk}
$T\cC\cap T\cA^{\omega}$ has constant rank if{f} the following compatibility conditions  between $E\rightarrow C$ and $D\rightarrow A$ are satisfied:
 \begin{eqnarray}
\label{ctcrk} &&D|_C \cap TC \text{ has constant rank}\\
\label{etark} && E\cap TA|_C \text{ has constant rank}\\
\label{prcond}
&& \text{The flows of vector fields }Y\in \Gamma(D)\subset \chi(A) \text{ with }Y|_C\in \chi(C)
\text{ preserves }E\cap TA|_C.
\end{eqnarray}
\end{lemma}
\begin{remark}  1) Condition \eqref{prcond} means that, assuming that $A/D$ is smooth and contains $C/(D|_C\cap TC)$ as a smooth submanifold, the projection
$pr \colon A\rightarrow A/D$ maps $E\cap TA|_C$ to a well-defined subbundle of $T(A/D)$.

2) In our later proofs we will make use only of the conditions \eqref{ctcrk},\eqref{etark}, \eqref{prcond}, and not of the fact that they derive from the constant rank condition on $T\cC\cap T\cA^{\omega}$.
\end{remark}

\begin{proof}
The constant rank 
condition on $T\cC\cap T\cA^{\omega}$ is stated algebraically by saying that the matrix $\{\phi_i,\psi_j\}\;\;mod\;\;\cI_{\cC}$
has constant rank,
where $\phi_i$ are generators  of
$\cI_{\cC}$ and $\psi_j$ generators   of $\cI_{\cA}$.

Assume that $T\cC\cap T\cA^{\omega}$ has constant rank. 
From the above characterization
 it follows that conditions \eqref{ctcrk} and \eqref{etark} hold.
This allows to choose the
generators  $\phi_i$ of $\cI_{\cC}$ as follows: degree $0$
generators $f_i\in Z(C)$ so that the last elements annihilate $E$,
degree $1$ generators $X_i\in \tGamma(E)$ so that the first elements
lie in $TC\cap E$. We choose the generators  $\psi_j$ of
$\cI_{\cA}\subset \cI_{\cC}$ to be compatible with the above choice
of $\phi_i$'s in the following sense: the degree $0$ generators $g_i\in Z(A)$ are so
that the last elements (restricted to C) annihilate $E$, the degree $1$
generators $Y_i\in \tGamma(D)$ so that the first elements (restricted
to C) lie in $TC\cap D|_C$. 

We write down the 4 by 4 block-matrix 

\begin{equation}
\{\phi_i,\psi_j\}\;\;mod\;\;\cI_{\cC}=
\label{matrix}
\left(\begin{array}{c c|c c}
0 &0&0&\alpha \\
0 &0&0&0 \\
\hline
0 &0&0&*\\
\beta &0&\delta&*
\end{array}\right).
\end{equation}
(The ``$0$'' in the lower right block comes from involutivity of the distribution $TC\cap E$). We have $$\alpha=\{f_i,Y_j\}\;\;mod\;\;\cI_{\cC}$$ where the differentials of $f_i\in Z(C)$ 
restricted to a complement of $E\cap TC$ in $E$ form a frame, and the $Y_j\in \tGamma(D)$ are a frame for a complement of $D|_C\cap TC$ in $D|_C$. Hence the \emph{columns} of the block $\alpha$ are linearly independent. Similarly we have
$$\beta=\{X_i,g_j\}\;\;mod\;\;\cI_{\cC}$$ where $X_i\in \tGamma(E)$ are a frame for a complement of $E\cap TC$ in $E$ and the differentials of $g_j\in Z(A)$  restricted to  a complement of $E\cap TA|_C$ in $E$ form a frame. Since $E\cap TC\subset E\cap TA|_C$ the \emph{columns} of $\beta$ are linearly independent too. 

Hence the above matrix \eqref{matrix} has constant rank if{f} the following condition is satisfied: $\delta$ consist of $C^{\infty}_1(\cC)$-linear combinations of the columns of $\beta$.

Reordering the vector fields $X_i\in \tGamma(E)$ used to define $\beta$ 
  so that the first few of them lie in $E \cap TA|_C$, we may assume that $\beta=\left( \begin{smallmatrix}  0\\ INV \end{smallmatrix}  \right)$ where $INV$ is an invertible matrix.
The above condition on
$\delta$ then means that the top rows of $\delta$ are identically zero, i.e., that
$$[X_i,Y_j]\subset \tGamma(E)$$
  where the $X_i\in\tGamma(E)$ span a complement of $E\cap TC$ in $E\cap TA|_C$
and $Y_j \in \tGamma(D)$ form a frame for $D|_C\cap TC$ at points of $C$.
Using  the fact that $Y_j|_A$ is a  vector field on $A$ which restricted to $C$ lies in the 
 integrable distribution $E\cap TC$, we see that this is equivalent to 
\begin{equation}\label{proje}
[\tGamma(E\cap TA|_C), Y_i]\subset \tGamma(E\cap TA|_C) \text{ for each }i,
\end{equation}
which is just condition \eqref{prcond}.
Conversely, if we assume conditions \eqref{ctcrk},\eqref{etark}, \eqref{prcond}, reversing the argument shows that $T\cC\cap T\cA^{\omega}$ has constant rank.
\end{proof}

{Before discussing when a function descends to the presymplectic quotient of $\cC$, we need a lemma.}
Consider the 
 Poisson algebra (see Remark \ref{moti} for a motivation)
$$\cE:=\cN(\cI_{\cA})\cap [f: \{f, \cI_{\cC}\cap \cN(\cI_{\cA})\}\subset \cI_{\cC}].$$

\begin{lemma}\label{lift}
Assume eq. \eqref{ctcrk},
\eqref{etark}, and \eqref{prcond}. Then, locally on $C$,
every class in 
$\cN(\cI_{\cC})/\cN(\cI_{\cC}) \cap\cI_{\cC}$ admits a representative which lies in  $\cN(\cI_{\cC})\cap \cE=\cN(\cI_{\cC})\cap \cN(\cI_{\cA})$. In particular
 we have
$$ \cN(\cI_{\cC})/(\cN(\cI_{\cC}) \cap\cI_{\cC} )\;\subset\; \cE/(\cE \cap\cI_{\cC}),$$
where both sides are seen as subsets of $C^{\infty}(\cC)$.
\end{lemma}
\begin{remark}\label{moti}
We explain why Lemma \ref{lift} is an algebraic version of Step (a) in the outline at the beginning of this Section; {to this aim, recall also Remark \ref{rem:charN}}. When $\bbC$ is smooth we have $C^{\infty}(\bbC)=\cE/(\cE \cap\cI_{\cC})$. 
Indeed $\cE$ consists exactly of
the functions on $\cM$ obtained as follows: start from a function on
$\bbC$, lift it to $\bar{\cC}$, extend it  to a function on 
$\bar{\cA}$ lying in the normalizer of $\cI_{ \bar{\cC}}$ (i.e., satisfying $\{F,\cI_{\bar{\cC}}\}_{\bar{\cA}}\subset \cI_{\bar{\cC}}$), then lift
to $\cA$ and extend in any way to $\cM$.
By Step (a) one expects
 a well-defined natural (bracket-preserving) map
$C^{\infty}(\un{\cC}) \rightarrow C^{\infty}(\bbC)$.
If $T\cA|_{\cC}\cap T\cC^{\omega}$ has constant rank (see Lemma \ref{lemmaconstrk}),
one expects further that every function on $\un{\cC}$ can be 
extended to a function on $\cM$ that annihilates both $T\cA$ and $T\cC^{\omega}$, 
i.e., by a function which lies both in $\cN(\cI_{\cC})$ and in $\cE$.
\end{remark}
\begin{proof}
We have to show that to  any $F\in \cN(\cI_{\cC})$ we can add some function in $\cI_{\cC}$ to obtain a function in $\cE\cap \cN(\cI_{\cC})$.
 We will make repeated use of eq. \eqref{Izero}, \eqref{Ione}, \eqref{NIzero}, \eqref{NIone}.

 To write out $\cE$ in classical terms notice that $\cN(\cI_{\cA})_1$ agrees with $\tGamma_{bas}(TA)$, the set of (extensions to $M$ of) vector fields on $A$ which are basic w.r.t. $pr \colon  A \rightarrow A/D$. We have
\begin{equation}
\cE_0=C_D^{\infty}(M)\cap  C_{(E\cap TA|_C)}^{\infty}(M),
\end{equation}
\begin{equation}
\cE_1=\{X\in \tGamma_{bas}(TA): [X,\bullet] \text{ preserves } \tGamma_{bas}(TA)\cap \tGamma(E\cap TA|_C)\}.
\end{equation} 
We have the following geometric set-up:  the projection $pr \colon A \rightarrow  A/ D=:\bar{A}$, and 
\[\xymatrix{  
E\cap TA|_C \ar[r]\ar[d] &pr_*(E\cap TA|_C)=:\bar{E} \ar[d]& \\
C \ar[r]  & C/(D|_C\cap TC)=:\bar{C} \ar[r] &  C/(E\cap TC)=\un{C} \\
}.
\]
Notice that $\bar{C}$ is well-defined by eq. \eqref{ctcrk}, and it can be regarded as a submanifold of $\bar{A}$ since we are working locally on $C$.
Also, the two vertical arrows are well-defined vector bundles by  \eqref{etark}  and \eqref{prcond}.
The (derivative of the) second quotient $\bar{C} \rightarrow \un{C}$ has  kernel $\bar{E}\cap T\bar{C}$.

To prove the lemma for degree $0$ functions take  $f\in \cN(\cI_{\cC})_0=C_E^{\infty}(M)$. The restriction $f|_C$ descends to a function on $\bar{C}$
annihilating $\bar{E}\cap T\bar{C}$. Extend  it to $A/D$ in such a way that annihilates $\bar{E}$.   Then pull back to a $D$-invariant function on $A$, and extend to a function on $M$ which annihilates $E$. This delivers an element of $\cE_0\cap \cN(\cI_{\cC})_0=C_D^{\infty}(M)\cap C_E^{\infty}(M)$ differs from $f$ by an element of $Z(C)=(\cI_{\cC})_0$.

To prove the lemma for degree $1$ functions take an element of $$\cN(\cI_{\cC})_1= \{X\in \tilde{\Gamma}(TC):[X,\tilde{\Gamma}(E)]\subset \tilde{\Gamma}(E)\}.$$ The restriction $X|_C$
descends to a vector field on $\un{C}$, which  we can lift  to a vector field $\bar{C}$. Lemma \ref{ext},
applied to the subbundle $\bar{E}\rightarrow \bar{C}\subset \bar{A}$, assures that we can further extend
to a vector field on $\bar{A}$ whose Lie derivative respects $\tGamma(\bar{E})$. Now   lift by $pr \colon A \rightarrow  A/ D=:\bar{A}$ to a vector field on $A$ tangent to $C$; its Lie derivative will automatically preserve $\tGamma(E\cap TA|_C)$. Then an extension to a vector field on  $M$  whose Lie derivative
preserves $\tGamma(E)$, which exists by Lemma \ref{ext},
gives an element $Y$ of $\cE_1\cap \cN(\cI_{\cC})_1$.
Since the projection $C \rightarrow  \un{C}$ maps $X|_C$ and $Y|_C$ to the same vector field on $\un{C}$, we conclude that $X-Y\in \tGamma(E)=(\cI_{\cC})_1$.

 Since $\cN(\cI_{\cC})/(\cN(\cI_{\cC}) \cap\cI_{\cC} )$ is generated by elements in degree $0$ and $1$ (see text after eq. \eqref{justy}) we are done.
\end{proof}

Now assume that $M$ is endowed with a Poisson structure $\pi$, corresponding to a function $\cS$ on $T^*[1]M$.

\begin{thm}\label{coisopre} Let $\cC$ be a presymplectic submanifold of $T^*[1]M$ whith smooth presymplectic quotient $\un{\cC}$.
Let $\cA$ be a graded coisotropic submanifold of $\cM$  containing
$\cC$ and assume  that eq. \eqref{ctcrk},
\eqref{etark}, and \eqref{prcond} hold. If
\begin{eqnarray}
 \label{SdescC} \cS&\in& \cN(\cI_{\cC})+\cI_{\cC},\\
 \label{SnormA}\{\cS,\cI_{\cA}\}&\subset& \cI_{\cC}\\
 \label{ShalfnormC}\{\cS, (\cI_{\cC}\cap\cN_{\cC}(\cI_{\cA}))_0\}&\subset&
 \cI_{\cC},
\end{eqnarray}
then ${\cS}$ descends to $\un{\cC}$  and
$\{\un{\cS},\un{\cS}\}=0$.

Here $\cN_{\cC}(\cI_{\cA})$ denotes the set of functions $F$ with $\{F,\cI_{\cA}\}\subset \cI_{\cC}$.
\end{thm}
\begin{remark}
Thm. \ref{coisopre} is slightly more general than the geometric construction outlined at the beginning of this section: since we are ultimately  interested in a
quotient of $\cC$, the requirement that $\cS|_{\cA}$ is constant
along every leaf of the characteristic distribution of $\bar{\cA}$
(condition  (b2) in the outline) is weakened to a condition on $\cC$ only (eq. \eqref{SnormA}).
\end{remark}

\begin{proof} 
Let $\un{f},\un{g},\un{h}$ functions on $\un{C}$ (which is smooth by Prop. \ref{pres2}), and fix extensions $\hat{f},\hat{g},\hat{h}$
 to functions in $\cN(\cI_{\cC})_0$.
Fix a choice of function
$\hat{\cS}\in \cN(\cI_{\cC})$ with $\hat{\cS}-\cS\in \cI_{\cC}$ (it exists by eq. \eqref{SdescC}). The proof of  Lemma \ref{halfnorm} shows that
\begin{equation}\label{keyeq}
\{\{\un{f},\un{g}\}_{\un{C}},\un{h}\}_{\un{C}}=\{\{\hat{f},\hat{g}\}_M, \hat{h}\}_M+ \{\{\cS,k\},\hat{h}\} \text{ mod } \cI_{\cC}
\end{equation}
where $k\in (\cI_{\cC})_0$ is explicitly given as
$$k=\{\{\hat{\cS}-\cS,\hat{f}\},\hat{g}\}.$$

Lemma \ref{lift}  allows us to refine the choice of extensions $\hat{\cS},\hat{f},\hat{g},\hat{h}$ so that they lie in $\cN(\cI_{\cC})\cap \cN(\cI_{\cA})$. In that case we claim that $k$ lies not only in $\cI_{\cC}$ but also in $\cN_{\cC}(\cI_{\cA})$. The proof goes as follows.
Since $\cN(\cI_{\cA})$ is a Poisson subalgebra,
$\{\{\hat{\cS},\hat{f}\},\hat{g}\}$ lies in $\cN(\cI_{\cA})\subset \cN_{\cC}(\cI_{\cA})$. Further
we have
$$\{\{\cS,{\hat{f}}\},\cI_{\cA}\}=\{\cS,\{\hat{f},\cI_{\cA}\}\}-
\{\hat{f},\{\cS,\cI_{\cA}\}\}\subset \{\cS,\cI_{\cA}\}+
\{{\hat{f}},\cI_{\cC}\}\subset \cI_{\cC},$$ where we used $\hat{f}\in
\cN(\cI_{\cA})\cap \cN(\cI_{\cC})$ and twice the fact that
$\cS \in \cN_{\cC}(\cI_{\cA})$ (which is just eq. \eqref{SnormA}).
Hence $\{\cS,\hat{f}\}\in\cN_{\cC}(\cI_{\cA})$, and taking a further bracket with $\hat{g}$
we conclude that $\{\{\cS,\hat{f}\},\hat{g}\}$ lies in $\cN_{\cC}(\cI_{\cA})$. Altogether this shows that
$k \in \cN_{\cC}(\cI_{\cA})$, as claimed.

From eq. \eqref{ShalfnormC}  it hence follows that $\{\cS,k\}\in
\cI_{\cC}$. Together with $\hat{h}\in \cN(\cI_{\cC})$ this implies that the whole second term of eq. \eqref{keyeq} lies in $\cI_{\cC}$. Taking
  the cyclic sum shows that  $\{\bullet,\bullet\}_{\un{C}}$  satisfies the Jacobi identity.
\end{proof}

Now we translate into classical terms Thm. \ref{coisopre}.
Until the end of this Section $(M,\pi)$ is a Poisson manifold, $C$ a submanifold, and
$E\subset TM|_C$ a subbundle such that $F:=E\cap TC$ is an involutive
constant rank distribution. Assume that $\un{C}:=C/(E\cap TC)$ is
smooth.
\begin{thm}\label{A1}
Let $A$ be a submanifold of $M$ containing $C$ and  $D$ an
 integrable distribution on $A$ such that
$D|_C\subset E$.
Assume eq. \eqref{ctcrk},
\eqref{etark}, and \eqref{prcond}, that means,
\begin{eqnarray*}
&&D|_C \cap TC \text{ has constant rank}\\
&&E\cap TA|_C \text{ has constant rank}\\ 
&& \text{the flows of vector fields }Y\in \Gamma(D)\subset \chi(A) \text{ with }Y|_C\in \chi(C)
\text{ preserves }E\cap TA|_C.
\end{eqnarray*}
 Assume that   the Poisson structure on $M$ satisfies
\begin{eqnarray}
\label{EEF} \{C^{\infty}_E(M),C^{\infty}_E(M)\}_M &\subset& C^{\infty}_{F}(M)\\
\label{frameD} (\cL_{X_i}\pi)|_C &\subset& E\wedge TM|_C\\
\label{ETCD}\sharp E^{\circ}&\subset& TC+D|_C
\end{eqnarray}
where the $X_i$ are obtained as follows: for any $p\in C$ make a choice of frame for $D|_C$ defined near $p$ and a choice of  extension to elements $X_i$ of $\tGamma(D)$.

Then $\un{C}$ is a Poisson manifold.
\end{thm}
\begin{proof}
As usual let $\cC=E^{\circ}[1]$, which is presymplectic with smooth presymplectic quotient by Prop. \ref{pres1} and \ref{pres2}.
We translate  in  terms of tensors on $M$ the three conditions on the function $\cS$ (given by the Poisson bivector $\pi$)
appearing in Thm \ref{coisopre}.
 The translation
 of condition \eqref{SdescC} is given in Lemma \ref{desc}:
 it corresponds to  \eqref{EEF} and $\sharp
E^{\circ}\subset TC+E$.
 Condition \eqref{SnormA}
 is equivalent to
\begin{eqnarray}
\label{NAE}\sharp E^{\circ}&\subset& TA|_C\\
\label{LieD} (\cL_{\tilde{\Gamma}(D)}\pi)|_C &\subset& E\wedge TM|_C.
\end{eqnarray}
 Condition \eqref{ShalfnormC} is equivalent
to \eqref{ETCD}.
We can forget \eqref{NAE} and $\sharp E^{\circ}\subset TC+E$ because
they are implied by \eqref{ETCD}.  Further
\eqref{LieD} can be replaced by \eqref{frameD}, i.e., it is enough to consider frames for $D|_C$. To see this notice that
any element of $\tGamma(D)$ is locally a $C^{\infty}(M)$-linear combination of the $X_i$ above and a vector field vanishing on $A$, and use \eqref{NAE}.

Hence all the assumptions of Thm \ref{coisopre} are satisfied, so
$\cS$ descends to a function on $\un{\cC}$ which commutes with
itself. Since by Prop. \ref{pres2} $\un{\cC}$ is symplectomorphic to
$T^*[1]\un{C}$, we obtain a Poisson structure on $\un{C}$.
\end{proof}
\begin{remark}\label{Liefun}
In Thm. \ref{A1} we can replace condition \eqref{frameD} by
$$\{C^{\infty}(M)_E \cap C^{\infty}(M)_D, C^{\infty}(M)_E \cap
C^{\infty}(M)_D\}\subset   C^{\infty}(M)_{D|_C},$$ as can be shown using
\eqref{ETCD}.
\end{remark}

\begin{remark}\label{replace}
We give a geometric description of $\bar{\cC}$,  the quotient  of $\cC$ by $T\cC \cap T\cA^{\omega}$. Recall from Rem. \ref{remC} (applied to $\cA$) that $\bar{\cA}$ is constructed as follows: divide the fibers of $D^{\circ}\rightarrow A$ to obtain $D^{\circ}/TA^{\circ}\rightarrow A$, which is endowed with
the flat $D$-connection $\nabla$;
 then identify fibers using $\nabla$. $E^{\circ}$ sits inside $D^{\circ}$, and after quotienting the fibers we obtain the subbundle $(E\cap TA)^{\circ}/TA^{\circ}$ of
$D^{\circ}/TA^{\circ}$.
Condition \eqref{prcond} is equivalent to $\nabla$ preserving this subbundle, and  identifying its fibers via $\nabla$ we obtain $\bar{\cC}$.
\end{remark}

A convenient application of  Thm. \ref{A1} is obtained
asking that $F\subset D|_C$ and that the submanifold $A$ be minimal (i.e.,
$TA|_C=TC+D|_C$). (One can show that these two requirements are equivalent to  $T\cC
+ T\cA^{\omega}= T\cA$, which in turns means that, locally, $\bar{\cC}=\bar{\cA}$.)

 \begin{thm}\label{A2}
Let $D|_C$ be a subbundle of $TM|_C$ with $F\subset D|_C\subset E$
and
\begin{equation}\label{ETCDdois}
\sharp E^{\circ}\subset TC+D|_C.
\end{equation}

Let $A$ be a submanifold containing $C$ such that $TA|_C=TC+D|_C$, and
 assume that  $D|_C$ can be
extended to an integrable distribution $D$ on $A$ such that
\begin{equation}
\label{Liedercon}
(\cL_{X_i}\pi)|_C \subset E\wedge TM|_C
\end{equation}
where the $X_i$ are obtained as follows: for any $p\in C$ make a choice of frame for $D|_C$ defined near $p$ and a choice of  extension to elements $X_i$ of $\tGamma(D)$.

Then $\un{C}$ is a Poisson manifold.
\end{thm}
\begin{proof}
We  check the assumptions of
Thm. \ref{A1}. We have $D|_C\cap TC=E\cap TC$ and $E\cap TA|_C=D|_C$, so conditions \eqref{ctcrk},\eqref{etark} and \eqref{prcond} are satisfied.
Condition \eqref{EEF} follows from eq. \eqref{frameD}. To see this use  Remark
\ref{Liefun},   $F\subset D|_C$, and the fact that  the bracket on $\un{C}$ is independent
of the choice of extension in $C^{\infty}_E(M)$ due to eq. \eqref{ETCDdois}. Since we are assuming conditions \eqref{frameD} and \eqref{ETCD}, we are done.
\end{proof}

\begin{remark}
The graded geometric approach gave a systematic way to derive and prove the statement of Thm. \ref{A2}.
A derivation without graded geometry seems hard. However once the statement is known there is a 
quick direct proof of Thm. \ref{A2} using algebraic arguments, which we wrote up in 
the Appendix of \cite{CZbilbao} 
\end{remark}

We conclude this Section by comparing Thm. \ref{A2} with the results of Falceto and the second author \cite{FZ}.
When $\sharp E^{\circ}\subset TC$ we can choose $D|_C=F$, so that
$A=C$. In this case Thm. \ref{A2} recovers\footnote{Use Remark \ref{Liefun} to see this.} exactly Prop. \ref{prescase}, which is also Prop. 4.1 of \cite{FZ},
and which is a mild extension of the Marsden--Ratiu  theorem \cite{MR}. 
 This also
shows that the submanifold $A$ in general does not have any Poisson
structure induced from $M$.

Further, by the ``minimality'' of $A$,  Thm. \ref{A2} improves Prop. 4.2 of \cite{FZ} because it weakens its assumptions .  Firstly, Prop. 4.2 of \cite{FZ} requires to extend $D|_C$ to an integrable distribution on the whole of $M$ (as opposed to just $A$). Secondly,  a technical assumption in Prop. 4.2 of \cite{FZ} -- on the compatibility between $E$ and the foliation defined by $D$ -- is now removed in Thm. \ref{A2}.

Several examples for  Thm. \ref{A2} have been discussed in \cite{FZ}. Here we limit ourselves to an example which, although trivial,    displays how Thm. \ref{A2} improves Prop. 4.2 of \cite{FZ}.
 
\begin{ep}
Let $(M,\pi)=(\RR^4, \sum_i \frac{\partial}{\partial x_i}\wedge
\frac{\partial}{\partial y_i})$ and $C=\{(x_1,y_1,0,0):x_1,y_1\in \RR\}$.  Let $E={\rm span}\{\frac{\partial}{\partial x_1},\frac{\partial}{\partial x_2},
\frac{\partial}{\partial y_2}+ x_1\frac{\partial}{\partial
y_1}\}$.  We have $
{\rm span} (\frac{\partial}{\partial x_1}- x_1\frac{\partial}{\partial
x_2})=
\sharp E^{\circ}\subset
TC+D|_C$ for $D|_C={\rm span}\{\frac{\partial}{\partial x_1},\frac{\partial}{\partial x_2}\}$.
Extending $D|_C$ to the hyperplane $A=\{y_2=0\}$ by the same formula clearly give an integrable distribution on $A$, and condition \eqref{Liedercon} is satisfied because
$\cL_{\frac{\partial}{\partial x_1}}\pi=\cL_{\frac{\partial}{\partial x_2}}\pi=0$.
Hence    Thm. \ref{A1} can be applied. (Since $\un{C}$ is one dimensional, the induced Poisson structure is of course trivial).

However the obvious extension of $D|_C$ to a distribution $\theta$ on the whole of $M$ does not satisfy the assumptions of Prop. 4.2 of \cite{FZ}. Indeed it is not compatible with the subbundle $E$, in the sense that
the projection $M \rightarrow M/\theta$ does not map the subbundle $E\subset TM|_C$ to a well-defined subbundle of $T(M/\theta)|_{\un{C}}$.
\end{ep}

\section{Examples} \label{sec:exdistr}

We present three examples of the Poisson reduction procedure we established.
The starting data is a manifold $M$ endowed with a (almost) Poisson structure, a submanifold $C$, and a subbundle $E\subset TM|_C$ such that $E\cap TC$ is an involutive distribution. (Actually in all three examples we have $E\oplus TC=TM|_C$, which implies that  $\cC:=E^{\circ}[1]$ is a symplectic submanifold of  $T^*[1]M$.). We describe the reduced structure both in the classical and in the graded picture.

To this aim we state the following lemma, which allows to compute easily the reduced Poisson structure in graded terms:
\begin{lemma}\label{onC}
Let $\cC$ be a presymplectic submanifold and $\cS$ a degree $2$ function on  $T^*[1]M$, so that $\cS|_{\cC}$ descends to the quotient $\un{\cC}$. Denote by $\pi^{\un{C}}$ the corresponding bivector field  on $\un{C}:=C/(E\cap TC)$.
Let $\cS'$ be a function such that $\cS'|_{\cC}=\cS|_{\cC}$ and so that  
$\cS'$ corresponds to a bivector field $\pi'$ on $M$ with $\pi'|_C\in \Gamma(\wedge^2 TC)$.
Then, writing   $pr \colon C \rightarrow \un{C}$, we have $$\pi^{\un{C}}=pr_*(\pi'|_C).$$
\end{lemma}
\begin{proof}
By the concrete form of the isomorphism $\cN(\cI)/(\cN(\cI)\cap \cI)\cong C^{\infty}(T^*[1]\un{C})$ (see
end of the proof of Prop. \ref{pres2}) it is clear that  the bivector field $pr_*(\pi'|_C)$ on $\un{C}$ corresponds to $\un{\cS'}$. Now just use $\un{\cS'}=\un{\cS}$.
\end{proof}


\begin{ep}\label{counterex}
Let $$(M,\pi)=(\RR^4,\partial_{x_1}\wedge \partial_{x_2}+\partial_{x_3} \wedge \partial_{x_4}),\;\;\;\;\; C=\{x_4=0\},\;\;\;\;\;E= \partial_{x_4}+\alpha \partial_{x_1}$$ for $\alpha=\alpha(x_1,x_2,x_3)\in C^{\infty}(C)$. This is Ex. 4.4 of \cite{FZ}, where the following is discussed: if one applies our Thm. \ref{A2} (or Prop. 4.2 of \cite{FZ}) with $A=M$ and as $D$ the extension of $E$ by translations in the $x_4$ direction, 
one sees that the induced bivector field on $C=\un{C}$ is Poisson provided  $\pd{x_1}\alpha=0$.

Here we present the corresponding graded picture. Denote by $\theta_i$ the degree $1$ coordinates on the fibers of $T^*[1]M$ conjugate to the $x_i$. The vanishing ideal $\cI_{\cC}$ of
$\cC=E^{\circ}[1]$   is generated by $x_4$ and $\theta_4+\alpha \theta_1$,
and $\cC$ is a symplectic submanifold of  $T^*[1]M$ since $\{x_4,\theta_4+\alpha \theta_1\}=-1$.
 We have
$\cS=\theta_1\theta_2+\theta_3\theta_4$.  Since  $\cS \notin \cN(\cI_{\cC})$, we perform reduction in stages to determine when $\cS|_{\cC}$ is a self-commuting function on $\un{\cC}=\cC$. Take the coisotropic submanifold $\cA$ with constraint $\theta_4+\alpha \theta_1$.
Notice that since $\cC$ is a symplectic submanifold and $\cC\subset \cA$, it follows that
$T\cC^{\omega}\cap T\cA|_{\cC}$ has constant rank, hence the same holds for $T\cC \cap T\cA|_{\cC}^{\omega}$. Further $\bar{\cA}=\bar{\cC}$. We check when $\cS \in \cN(\cI_{\cA})$: a computation shows $\{\cS,\theta_4+\alpha \theta_1\}=-[\pd{x_1}\alpha\cdot  \theta_2+\pd{x_3}\alpha \cdot \theta_4]\theta_1$, which lies in $\cI_{\cA}$ if{f} $\pd{x_1}\alpha=0$. In this case $\cS|_{\cA}$ descends to $\bar{\cA}=\bar{\cC}=\cC$ and commutes with itself there.

Since on $\cC$ we have $\theta_4=-\alpha \theta_1$,
a function with $\cS'|_{\cC}=\cS|_{\cC}$ which corresponds to a bivector field on $C$ 
is $\cS':=\theta_1(\theta_2+\alpha \theta_3)$, so by Lemma \ref{onC} we deduce that the induced Poisson bivector field on $C$ is $\pd{x_1}\wedge (\pd{x_2}+\alpha \pd{x_3})$.

We end this example displaying explicitly an extension of $\cS|_{\cC}$ to a function $\hat{\cS}$ lying
in $\cN(\cI_{\cC})$. To this aim we modify $\cS'$ to obtain an element
 $\hat{\cS}:=\cS'-x_4\{\cS',\theta_4+\alpha \theta_1\}$. Notice that the commutator
$\{\hat{\cS},\hat{\cS}\}$ lies in $\cI$ if{f}\footnote{Use $\{\cS',x_4\}=0$ and $x_4 \in \cI$ to see this.}  $\{\cS',\cS'\}=-2 (\partial_{x_1} \alpha) \theta_1 \theta_2 \theta_3$ does. Hence we see again that, unless $\alpha$ is independent of $x_1$,  $\cS$ descends to
a function on $\un{\cC}=\cC$ which does not commute with itself.
\end{ep}

We show that   every almost Poisson manifold may be
obtained by reduction from its cotangent bundle with canonical
symplectic structure.

\begin{ep}
Let $C$ be a manifold endowed with a
 bivector field $\alpha$ (not necessarily Poisson).
 Embed $C$ as the zero section in
$M:=T^*C$, which is endowed with the canonical symplectic structure $\omega_M$. Let
$$E=graph(-\frac{1}{2}\sharp_C)\subset TM|_C=TC\oplus T^*C$$ where  $\sharp_C: T^*C \rightarrow TC$ denotes contraction with $\alpha$. We claim that the bivector field  induced by the reduction of the triple
$(M,C,E)$
on  $\un{C}=C$ is again $\alpha$.

In classical terms we know that the induced bracket on $C$ is computed
extending functions on $C$ to elements of $C_E^{\infty}(M)$, taking their bracket with respect to the   $\omega_M$, and restricting to $C$ (see Prop. \ref{prescase}). This can be computed choosing base coordinates $x_i$ on $C$ and conjugated fiber coordinates $y_i$ on $M=T^*C$, so that $E$ is spanned by the $\pd{y_i}-\frac{1}{2}\sum_j \alpha_{ij}\pd{x_j}$ (where $\alpha=\frac{1}{2}\alpha_{ij}\pd{x_i}\wedge \pd{x_j}$). We extend the functions $x_j$ on $C$ to the functions
$\tilde{x}_j:=x_j+
\frac{1}{2}\sum_i \alpha_{ij}y_i \in C_E^{\infty}(M)$. Then $(\{\tilde{x}_j,\tilde{x}_{j'}\}_M)|_C=\alpha_{jj'}$ as claimed.

This can be seen more easily in the graded picture.
Take the  degree $1$ coordinates $Q_i$ and $P_i$   on the fibers of $T^*[1]M$, conjugated to $x_i$ and $y_i$.
The graded submanifold $\cC=E^{\circ}[1]$ of $T^*[1]M$ is given locally by the constraints $$ y_i=0 \;\;\;\;\;\;\;\;P_i-\frac{1}{2}\alpha^{ij}Q_j=0 \;\;\;\;\;\;\text{   for all }i.$$ The matrix obtained taking the Poisson bracket of these constraints in $T^*[1]M$
is   non-degenerate, so $\cC$ is a graded symplectic submanifold, canonically symplectomorphic to $T^*[1]C$.
 The canonical symplectic structure $\omega_M$ on $M=T^*C$ corresponds to the function $\cS=Q_iP_i$ on $T^*[1]M$, and $\cS|_{\cC}=\frac{1}{2}\alpha^{ij}Q_iQ_j$, so using Lemma \ref{onC} we see that it recovers the bivector field $\alpha$ on $C$ we started with.
\end{ep}

We show that if $\g$ is a Lie \emph{bi}algebra then the Poisson structure on its dual can be recovered from the Poisson structure on the dual of the Drinfeld double.
\begin{ep}
Let $\g$ be Lie bialgebra. Endow $\g\oplus \g^*$ with the Drinfeld double Lie algebra structure \cite{Lu}(which also contains cross-terms). Its dual  $M:=(\g\oplus \g^*)^*=\g^*\oplus \g$ is a linear Poisson manifold.
Let $$C:= \g^* \oplus \{0\} \subset \g^*\oplus \g.$$
Let $E_{(x,0)}=\{0\}\oplus \g \subset T_{(x,0)}M$ for all $x\in \g^*$.
The reduction  of the triple
$(M,C,E)$ recovers the natural Poisson\footnote{$C$ has a natural Poisson structure even thought it is not a Poisson submanifold of $M$ (it is just a coisotropic submanifold).}
structure on $\un{C}=C$, namely the one obtained as the dual of the  Lie algebra $\g$. Since the subbundle $E$ is transversal to $C$, this is seen simply extending linear functions on $C$ to linear functions on $M$ that annihilate $E$.

The assumptions of our Prop. \ref{prescase}, which is essentially the Marsden--Ratiu theorem \cite{MR}, are not satisfied, since contraction with the Poisson bivector of $M$ does not map $E^{\circ}$ into $TC$ 
(essentially due to the fact that the Drinfeld double bracket contains crossed terms).
Thm. \ref{A2} however applies, with $A=M$ and  $D$ equal to the extension of $E$ from $C$ to $M$ by translation\footnote{To see that the assumptions of Thm. \ref{A2} are satisfied use  Remark \ref{Liefun}.}. 

The graded picture is as follows. $\cC:=E^{\circ}[1]$ coincides with the canonical  copy of $T^*[1]C$ inside $T^*[1]M$, hence it is symplectic. If we choose a basis of $\g$ and the dual basis, we obtain coordinates $x_i$ on $\g^*$ and $y_i$ on $\g$, hence coordinates on $M$. Denote by $P_i$ and $Q_i$ respectively the conjugated degree $1$ coordinates on the fibers $T^*[1]M$. Then $\cC$ is given by setting to zero all $y$ and $Q$, and the degree $2$ function encoding the Poisson structure on $M$ is
$$\cS=\frac{1}{2}\left[c_{ij}^kx_kP^iP^j+d_{ij}^ky_kQ^iQ^j+(c_{ik}^jy_k-d_{jk}^ix_k)P^iQ^k \right].$$
The restriction to $\cC$ consists just of the first summand above, and recovers the natural Poisson structure on $C$.

We check explicitly that $\cS|_{\cC}$ commutes with itself.
One computes $\{S,y_l\}  \notin \cI_{\cC}$, so  $\cS \notin \cN(\cI_{\cC})$. On the other hand reduction in stages using the coisotropic submanifold  $\cA=D^{\circ}[1]=\{Q=0\}$ applies. Indeed $\cS \in \cN(\cI_{\cA})$ since
$\{\cS,Q_l\}=\frac{1}{2}[d_{ij}^lQ^iQ^j+c_{il}^jP^iQ^l],$ so restricting $\cS$ we obtain a function on $\bar{\cA}=\bar{\cC}=\cC$ with commutes with itself.
\end{ep}

\newpage

\part{Reduction of Poisson manifolds by group actions}

We saw in Prop. \ref{dimitri} that the data of a Poisson manifold
$(M,\pi)$ can be encoded in the degree $1$ symplectic manifold
$\cM:=T^*[1]M$ together with a degree $2$ function $\cS$ on $\cM$
satisfying $\{\cS,\cS\}=0$. When considering actions on $\cM$ one
wishes to take into account its  symplectic structure and the function
$\cS$, {and in particular the hamiltonian vector field $X_{\cS}=\{\cS,\cdot\}$ (which is nothing else than the Lichnerowicz differential that defines Poisson cohomology).}
Notice  the graded Lie algebra of vector
fields $\chi(\cM)$ is endowed with the differential $[X_{\cS},\cdot]$, 
making $\chi(\cM)$ into a differential graded Lie algebra (DGLA).
Hence it is natural to act infinitesimally on $\cM$ by a DGLA, rather
than just by a graded Lie algebra, and to ask that the action be
hamiltonian (in the sense of ordinary symplectic geometry).

We will consider such an action on $\cM$, construct its global and
Marsden--Weinstein quotients, and translate them into classical terms
(Sections \ref{actions}--\ref{sec:hamred}).
Further, we will see that there is an induced action on the symplectic
groupoid $\Gamma$  of $(M,\pi)$, we show that  interestingly it is an action
in the category of Lie groupoids, and construct its global and
Marsden--Weinstein quotients (Sections \ref{sec:groupGamma}--\ref{sec:cga}).
\section{Actions on $T^*[1]M$}\label{actions}

Let $M$ be a manifold, and consider the degree $1$ symplectic manifold $\cM:=T^*[1]M$. In this Section we introduce the infinitesimal actions on $\cM$ that we will study in the rest of the paper.

An {infinitesimal action} on $\cM$ is a morphism of graded Lie algebras into $\chi(\cM)$, the graded Lie algebra of vector fields on $\cM$. Since vector fields of degrees other than $-1$ and $0$ vanish on the body $M$, any graded Lie algebra with a locally free infinitesimal action on $\cM$ must be concentrated in degrees $-1$ and $0$. Hence we are lead to  consider a graded Lie algebra $(\h[1] \oplus \g,[\cdot,\cdot])$ , where $\h$ and $\g$ are usual (finite dimensional) vector spaces.
We suppose that it acts infinitesimally on   $T^*[1]M$ \emph{with moment map}.
This means that we have a   morphism of graded Lie algebras 
into the hamiltonian vector fields of $\cM$, i.e. one of the form

\begin{center}\fbox{
\begin{Beqnarray*}
 \psi \colon \h[1] \oplus \g &\rightarrow& {\chi}_{-1}(\cM)\oplus  {\chi}_0(\cM)\\
(w,v)&\mapsto& \;\;\;X_{{J_0}^*w}+ X_{{J_1}^*v}.
\end{Beqnarray*}}
\end{center}

The notation is chosen so that ${{J_0}^*w}\in C^{\infty}_0(\cM)$ and ${{J_1}^*v}\in C^{\infty}_1(\cM)$. 
We spell out the data we assumed. $\h[1] \oplus \g$ being a graded Lie algebra means that $\g$ is a (ordinary) Lie algebra and $\h$ a $\g$-module (where $v\in \g$ acts as $[v,\bullet]$).  ${J_0}^*$ is a linear map $\h \rightarrow C^{\infty}_0(\cM)=C^{\infty}(M)$, so it corresponds to a map $J_0 \colon M \rightarrow \h^*$, while 
${J_1}^*$ is a linear  map $\g \rightarrow C^{\infty}_1(\cM)=\chi(M)$. We do not spell out the requirement that 
$\psi$ be a morphism of graded Lie algebras, since it is implied by the requirement that 
 the moment map be a Poisson map.

\begin{lemma}\label{poismap}
The moment  map $({J_0},{J_1}): \cM \rightarrow    (\h[1] \oplus \g)^*[1]$
is a Poisson map  if{f}
\begin{itemize}
\item [1)]  ${J_1}^*$ is an infinitesimal action of $\g$ on $M$
 \item [2)]  ${J_0}:  M \rightarrow \h^*$ is $\g$-equivariant, i.e.,
${J_1}^*v({J_0}^*w)={J_0}^*([v,w])$ for all $v\in \g$ and $w\in \h$.
\end{itemize}
\end{lemma}
\begin{proof}
We denote elements of $\g$  by $v$ and elements of $\h$  by $w$. The equation $\{{J_1}^*v,{J_1}^*\tilde{v}\}={J_1}^*[v,\tilde{v}]$
means by definition that ${J_1}^*$ is an infinitesimal action.
The analog equation for elements of $\h$ is satisfied because both sides vanish by degree reasons.
We have $\{{J_1}^*v,{J_0}^*w\}={J_1}^*v({J_0}^*w)$, where in the second term ${J_1}^*v$ is viewed as 
a vector field on $M$ acting on a function.
This is equal to
${J_0}^*([v,w])$   if{f} ${J_0}$ is $\g$-equivariant (recall that  the action of $\g$ on $\h$ is given by 
$v\mapsto [v,\bullet]$). 
\end{proof}

Summarizing we have:
\begin{cor} Let $M$ be a manifold.
An \emph{action of the graded Lie algebra $\h[1] \oplus \g$   on
$\cM=T^*[1]M$  so that the moment map is Poisson} corresponds exactly
to the following classical data:
\begin{itemize}
\item  a (ordinary) Lie algebra $\g$
\item a linear action of $\g$ on the vector space $\h$
\item an infinitesimal action of $\g$   on $M$ (denoted by  $J_1^*$)
\item a $\g$-equivariant map ${J_0}:  M \rightarrow \h^*$.
\end{itemize}
\end{cor}
   
{Until the end of this paper we will consider only actions on $\cM$ with \emph{Poisson} moment map.}\\   
   
Now we introduce   new pieces of data. We assume that $M$ is endowed with a Poisson
 structure $\pi$, i.e  that $\cM$ is endowed with a self-commuting degree $2$
 function $\cS$ (Prop. \ref{dimitri}). Then 
$$(\chi(\cM), [\cdot,\cdot], [X_{\cS},\cdot])$$ is a differential graded Lie algebra (DGLA, see Def. \ref{dla}).
Hence it is natural to assume that $\cM$ is acted upon not by a graded Lie algebra, but rather by a DGLA
$$(\h[1] \oplus \g, [\cdot,\cdot], \delta).$$

\begin{remark} DGLAs concentrated in degrees $-1$ and $0$ are in bijective correspondence to crossed modules of Lie algebras, see Appendix \ref{A}.
\end{remark}

\begin{lemma}\label{differ}
Suppose that $\psi \colon \h[1] \oplus \g \rightarrow {\chi}(\cM)$ is 
a linear map of the form $(w,v)\mapsto \;\;\;X_{{J_0}^*w}+ X_{{J_1}^*v}$.
Then $\psi$ respects differentials if{f}
\begin{itemize}
\item[1)] ${J_1}^*(\delta w)=X^{M}_{{J_0}^*w}$ for all $w\in \h$
\item[2)] The $\g$ action on $M$ is by Poisson vector fields.
\end{itemize}
\end{lemma}
\begin{proof}
As earlier we denote elements of $\g$  by $v$ and elements of $\h$  by $w$. The equation $[X_{\cS},X_{{J_0}^*w}]=X_{{J_1}^*(\delta w)}$ holds if{f} $\{\cS,{J_0}^*w\}={J_1}^*(\delta w)$, which is just 1).

Since $\delta v=0$, we have to check when
$[X_{\cS},X_{{J_1}^*v}]$ vanishes. It vanishes if{f} $\{\cS,{J_1}^*v\}=-\cL_{{J_1}^*v}\pi$ vanishes, giving 2).
\end{proof}

We summarize:  
\begin{cor}\label{summ}
Let $(M,\pi)$ be a Poisson manifold and $\cM=T^*[1]M$. 
A \emph{structure of a DGLA on $\h[1] \oplus \g$ together with
a morphism of DGLAs  $\psi \colon \h[1] \oplus \g \rightarrow {\chi}(\cM)$ with Poisson moment map} is equivalent to the following classical data\footnote{The first three items encode the fact that $\h[1] \oplus \g$ is a DGLA concentrated in degrees $-1$ and $0$, and can equivalently be  expressed saying that $\h$ and $\g$ form a crossed module of Lie algebras, see Appendix \ref{A}.}:
\begin{itemize}
\item  a (ordinary) Lie algebra $\g$
\item a linear action of $\g$ on the vector space $\h$
\item a linear map  $\delta \colon \h \rightarrow \g$ such that $\delta(v\cdot w)=[v,\delta w]$ and $(\delta w)\cdot \tilde{w}=-(\delta \tilde{w})\cdot w$, for all $v\in \g$ and $w,\tilde{w}\in \h$ 

\end{itemize}
together with
\begin{itemize}
\item an infinitesimal action $J_1^*$ of $\g$   on $(M,\pi)$ by Poisson vector fields
\item a $\g$-equivariant map ${J_0}:  M \rightarrow \h^*$
\end{itemize}
satisfying the following ``moment map condition'' for all $w\in \h$:$${J_1}^*(\delta w)=X^{M}_{{J_0}^*w}.$$ 
\end{cor}
 
See Section \ref{secex} for examples of morphisms of DGLA with Poisson moment map.\\

We comment on when ${J_0}:  M \rightarrow \h^*$ is a Poisson map. Here $\h$ is endowed with the Lie algebra structure $[w,\tilde{w}]_{\delta}:=[\delta w,\tilde{w}]$ (see also the text after Lemma \ref{dglaxm}), and $\h^*$ has the corresponding linear Poisson structure.

\begin{lemma}\label{mopois} Let $\psi$ be as in Lemma \ref{differ}. 
$J_0 \colon (M, \pi)\rightarrow \h^*$ is a Poisson map 
if{f} $J_0$ is  equivariant under the action of the Lie subalgebra $\delta
 (\h)\subset \g$.
\end{lemma}
\begin{proof} For all $w,\tilde{w}\in \h$ we 
 have $$\{J_0^*w,J_0^*\tilde{w}\}_M=X_{J_0^*w}^M(J^*_0\tilde{w})=J^*_1(\delta w)(J^*_0\tilde{w}),$$
where we used Lemma \ref{differ} 1) in the second equality. Further, by definition of $[\bullet,\bullet]_{\delta}$, 
 $$J^*_0[w, \tilde{w}]_{\delta}=J^*_0[\delta w, \tilde{w}].$$ 
$J_0$ being a Poisson map corresponds to  equality of the left-most terms,
and $J_0$ being $\delta (\h)$-equivariant corresponds to equality of the right-most terms.
\end{proof}

\begin{remark}
If $\psi$ is a morphism of DGLAs   with Poisson moment map, then $J_0$ is  a Poisson map.
This is clear from the  Lemma \ref{mopois} together with  Lemma \ref{poismap}.
\end{remark}

In the rest of this section we use the correspondence between degree $1$ $N$-manifolds and vector bundles in order to interpret   in classical terms the action $\psi$.  {The following interpretation of  vector fields appears in the literature, see for instance in \cite{vorQ}\cite{MR2768006}. We provide a proof for completeness.}
\begin{lemma}\label{inter}
Let $A\rightarrow M$ be a vector bundle. 
 There are canonical, bracket preserving identifications 
\begin{eqnarray*} 
{\chi}_{-1}(A[1])&\cong &\{\text{vertical vector fields on $A$  constant on each fiber}\}\\
 {\chi}_{0}(A[1])&\cong& \{\text{vector fields on }A\text{ preserving its vector bundle structure}\}.
\end{eqnarray*}
\end{lemma}
\begin{proof} We have
$C^{\infty}_0(A[1])=C^{\infty}(M)$ and $C^{\infty}_1(A[1])=\Gamma(A^*)=
C^{\infty}_{lin}(A)$, where the latter denotes the fiber-wise linear functions on $A$.
${\chi}_{-1}(A[1])\oplus {\chi}_0(A[1])$ acts on $C^{\infty}_{-1}(A[1])\oplus C^{\infty}_0(A[1])$, hence it also acts (by non-graded derivations) on the functions on $A$ which are at most linear on the fibers. Since
functions on $A$ which are at most linear on the fibers generate (up to completion issues) $C^{\infty}(A)$, we conclude that we can view vector fields on $\cM$ of degree $-1$ and $0$ as vector fields on $A$, and that this correspondence preserves the Lie bracket operation on vector fields.

Concretely,  ${\chi}_{-1}(A[1])$ acts by annihilating $C^{\infty}(M)$ and acts on $\Gamma(A^*)$ by contraction with an element of $\Gamma(A)$. ${\chi}_0(A[1])$ acts preserving the fiberwise constant and the fiberwise linear functions on $A$. 
\end{proof}

\begin{remark}\label{gmVB}
If $A=T^*M$, so that $A[1]=\cM$, the  identification of Lemma \ref{inter} restricts to:
 \begin{eqnarray*}
 \forall  F\in C^{\infty}_0(\cM)=C^{\infty}(M)&:& X_F\in  {\chi}_{-1}(\cM) \mapsto \text{the constant extension of } -dF\in \Omega^1(M)\\
  \forall Y\in  C^{\infty}_1(\cM)=\chi(M)\;\;\;\;&:& X_Y\in {\chi}_{0}(\cM) \;\;\mapsto \text{the cotangent lift of }Y.
\end{eqnarray*}
Hence an infinitesimal  action $\psi$ of the graded Lie algebra $\h[1]\oplus \g$ in $\cM=T^*[1]M$ can be also viewed as an infinitesimal action  of the Lie algebra $\h \oplus \g$ (with the ``same'' bracket as the one of the graded Lie algebra $\h[1]\oplus \g$) on $T^*M$: 
\begin{center}\fbox{
\begin{Beqnarray*}
 \hat{\psi} \colon \h \oplus \g &\rightarrow& {\chi}(T^*M)\\
(w,0)&\mapsto& \;\;\; -d({J_0}^*w) \text{ viewed as  vertical v.f.  constant on fibers}\\
(0,v)&\mapsto& \;\;\; \text{the cotangent lift of }{J_1}^*v
\end{Beqnarray*}}
\end{center}
\end{remark}

\section{The global quotient of  $T^*[1]M$}

In this section we consider quotients of $\cM=T^*[1]M$ by actions of graded Lie algebras or   DGLAs.

Let $\h[1] \oplus \g$ be a graded Lie algebra. This is equivalent to saying the $\g$ is a Lie algebra and  $\h$ is a $\g$-module (see Section \ref{actions}).
Let $G$ be the simply connected Lie group integrating the Lie algebra $\g$. Let $\rho$ denote the Lie group action of $G$ on $\h$ integrating the action of $\g$ on $\h$.

\begin{lemma}
The graded Lie group integrating the graded Lie algebra   $\h[1]
\oplus \g$ is $G\ltimes \h[1]$, with product 
\begin{equation}\label{Gh1}
(g_1,w_1)\cdot
(g_2,w_2)=(g_1g_2,w_1+\rho(g_1)w_2).
\end{equation}
 \end{lemma}
\begin{proof}
The graded Lie bracket on $\h[1] \oplus \g$ makes the vector space $\h \oplus \g$ into an honest Lie algebra, namely the semidirect product Lie algebra \cite[Sect. 3.14]{varLG} of $\g$ and $\h$ (viewed as an   abelian Lie algebra) by the action. This semidirect product Lie algebra integrates to the semidirect product Lie group of $G$ and the abelian Lie group $\h$ by the action $\rho$; its multiplication  is given by eq. \eqref{Gh1} (see \cite[Sect. 3.15]{varLG}). The lemma follows since there is a  canonical identification between degree 1 N-manifolds and vector bundles (see Lemma \ref{inter}), and 
the construction of the graded Lie algebra of a graded Lie group is given by phrasing  the construction of the Lie algebra of an ordinary Lie group
in algebraic terms \cite[Sect. 7.1]{var}. \end{proof}

Now let $\psi \colon \h[1] \oplus \g \rightarrow {\chi}(\cM)$
be an infinitesimal action.
The infinitesimal action $\psi$ integrates\footnote{We introduce $\Psi$ because, being a global action, it is geometrically more appealing than its infinitesimal counterpart $\psi$. In what follows (Prop. \ref{glocM} and Prop. \ref{MWcm}) $\Psi$ can be replaced by the infinitesimal $\psi$ since the Lie group $G$ is connected.} to a left action
\begin{equation*}
\boxed{{\Psi}\colon (G\ltimes \h[1]) \times \cM \rightarrow \cM}.
\end{equation*}
 
 \begin{remark}\label{gmVB2} 
We saw in Remark \ref{gmVB} that the graded infinitesimal action $\psi$   can be also viewed as an infinitesimal action $\hat{\psi}$ of a (ordinary) Lie algebra on $T^*M$. 
Similarly, the action $\Psi$ corresponds to the integration of $\hat{\psi}$, i.e., to 
the following action of $G\ltimes \h$ on $T^*M$: $G$ acts by the cotangent lift of its action on $M$, and $w\in \h$ acts
translating in the fiber directions by 
$-d({J_0}^*w)$.
\end{remark}

Now we determine the quotient of $\cM$ by the action $\Psi$, defined as the graded manifold whose algebra of functions is isomorphic to \begin{equation}\label{inva}
 C^{\infty}(\cM)^{\Psi}= C^{\infty}(\cM)^{\psi}:=\{F\in C^{\infty}(\cM): \psi(z)(F)=0 \text{ for all }z\in  \h[1] \oplus \g \}.
\end{equation}

\begin{prop}\label{glocM}  
Let $\h[1] \oplus \g$ be a graded Lie algebra. Let $\psi \colon \h[1] \oplus \g \rightarrow {\chi}(\cM)$
be an infinitesimal action with Poisson moment map $(J_0,J_1)$. Assume that $J_0 \colon M \rightarrow \h^*$ is a submersion and that 
the $G$-action on $M$ obtained integrating $J_1^* \colon \g \rightarrow \chi(M)$ is free and proper. 
 Then  the quotient of $\cM$ by the $\Psi$ action is smooth, and 
there is a canonical isomorphism  $$\cM/(G\ltimes \h[1])\cong (D/G)^*[1]$$
where $D:=ker (J_0)_* \subset TM$.

Assume further that $\h[1] \oplus \g $ is a DGLA, $(M,\pi)$ a Poisson manifold, and that $\psi$ is a morphism of DGLAs.
Then $(D/G)^*[1]$  inherits a degree $-1$ Poisson structure and a homological Poisson vector field. Hence $D/G\rightarrow M/G$ is a Lie \emph{bi}algebroid.
\end{prop}

{Recall that a {vector field} $Q$ on a graded manifold is \emph{homological} if{f} it has degree 1 and $[Q,Q]=0$. A graded manifold together with  a homological vector field constitutes a Q-manifold \cite{AKSZ}.}
In the  proof we will use the following characterizations of the notions of 
Lie algebroid
and Lie bialgebroid {(see \cite{MK2} for the definition  of Lie (bi)algebroid).}
\begin{remark}\label{bial}
 A \emph{Lie algebroid} is a vector bundle $A$ such that $A[1]$
is endowed with a homological vector field, i.e., a degree $1$ vector field $Q_A$ with $[Q_A,Q_A]=0$ \cite{vaintrob}.
Equivalently, it is a vector bundle $A$ such that $A^*[1]$ is endowed with a degree $-1$ Poisson structure \cite[Section 4.3]{yvette}. A \emph{Lie bialgebroid} is a vector bundle $A$ such that 
$A^*[1]$ is endowed with a homological vector field $Q_{A^*}$ and a degree $-1$ Poisson structure $\{\cdot,\cdot\}$, compatible in the sense that  $Q\{X,Y\}=\{Q(X),Y\}+\{X,Q(Y)\}$ for all $X,Y \in \Gamma(A)=C^{\infty}_1(A^*[1])$. The compatibility condition can be rephrased saying that $Q$ is a Poisson vector field on $A^*[1]$.
\end{remark}

\begin{proof}
$C^{\infty}_0(\cM)^{\Psi}$ agrees with $C^{\infty}(M)^G$, since the $G$ action on $M$ is obtained integrating the vector fields $J^*_1v$ for $v\in \g$. 
$C^{\infty}_1(\cM)^{\Psi}$ consists of sections of $D$ which are invariant under the $G$ action on $M$.
Notice that the $G$-action on $M$ preserves $D$: elements $X\in \Gamma(D)\subset \chi(M)$ are characterized by $\langle d(J^*_0w), X \rangle =0$ for all $w\in \h$, and for all $v\in \g$ we have
$$\langle d(J^*_0w), \cL_{J^*_1v} X \rangle = -\langle \cL_{J^*_1v} d(J^*_0w), X \rangle=-\langle d(J^*_0[v,w]),X  \rangle=0$$
where the second equality holds by Lemma \ref{poismap} 2).

Our assumptions imply that the infinitesimal action $\psi$ is  locally free in the following sense:
for each $p\in M$ the linear map $\h[1] \oplus \g \rightarrow  T_p\cM$, obtained from $\psi$ evaluating at $p$, is injective. Using \cite[Lemma 4.7.3]{var}
it follows
that the image of   $\psi$ spans a distribution  
on $T^*[1]M$.  
Further it is an involutive distribution since the map $\psi$ preserves brackets.
By the above description the degree $0$ and $1$ invariant functions satisfy assumptions $i)$ and $ii)$ of Lemma \ref{Redu}. Using that Lemma and the fact that $D/G \rightarrow M/G$ is a smooth vector bundle we conclude that 
 $C^{\infty}(\cM)^{\Psi}\cong C^{\infty}((D/G)^*[1])$, proving the first statement.

To prove the second part of the lemma notice that
$C^{\infty}(\cM)^{\Psi}$ is a graded Poisson subalgebra of $C^{\infty}(\cM)$ since the infinitesimal action $\psi$ acts by Poisson vector fields. Hence $(D/G)^*[1]$
is endowed with a degree $-1$ Poisson structure.
The homological vector field  $Q=X_{\cS}=\{\cS,\bullet\}$
preserves\footnote{The special case that $\cS \in  C^{\infty}(\cM)^{\Psi}$ occurs if{f} the differential $\delta \colon \h \rightarrow \g$ is trivial. In this case the fibers of $J_0$ are Poisson submanifolds of $(M,\pi)$.}
 $C^{\infty}(\cM)^{\Psi}$: if $X$ is an infinitesimal generator of the action and $F\in  C^{\infty}(\cM)^{{\Psi}}$     then
$$X(Q(F))=(-1)^{|X|}Q(X(F))+[X,Q](F)=0,$$ because
the assumption that  $\psi$ is a morphism of DGLAs ensures that $[X,Q]$ is also an infinitesimal generator of the action. So $Q$ induces a homological vector field $\un{Q}$ on 
$(D/G)^*[1]$. $\un{Q}$ is a Poisson vector field, since $Q$ is. Hence, by Remark \ref{bial}, $D/G$ is a Lie bialgebroid.
\end{proof}

\section{Hamiltonian reduction of $T^*[1]M$}\label{sec:hamred}

In this Section  we perform Marsden--Weinstein reduction at zero for the hamiltonian action $\Psi$ on $\cM=T^*[1]M$.

\begin{prop}\label{MWcm}
Let $\h[1] \oplus \g$ be a graded Lie algebra. Let $\psi \colon \h[1] \oplus \g \rightarrow {\chi}(\cM)$
be an infinitesimal action with Poisson moment map $(J_0,J_1)$. Assume that zero is a regular value of $J_0 \colon M \rightarrow \h^*$ and that the $G$-action on $C:=J_0^{-1}(0)$ obtained integrating\footnote{Since $J_0$ is $\g$-equivariant by Lemma
 \ref{poismap} 2), $G$ acts on $C\subset M$.}
$J_1^* \colon \g \rightarrow \chi(M)$ is free and proper.
Then  
 $\cC:=({J_0},{J_1})^{-1}(0)$ is smooth,  
the Marsden--Weinstein reduced space  $\cC/(G\ltimes \h[1])$ is a degree $1$ symplectic manifold, and we have a   canonical symplectomorphism
 $$\cC/(G\ltimes \h[1]) \cong T^*[1](C/G).$$

Assume further that $\h[1] \oplus \g $ is a DGLA, $(M,\pi)$ a Poisson manifold, and that $\psi$ is a morphism of DGLAs. Then the function  $\cS\in C^{\infty}_2(\cM)$ corresponding to $\pi$ descends to a self-commuting function on 
 $\cC/(G\ltimes \h[1])$. Hence $C/G$  has an induced Poisson structure.
 \end{prop}
 
\begin{remark}
When $\h=\{0\}$, the first part of Prop. \ref{MWcm} specializes (up to the degree shift by one) to the following classical statement: given a free and proper action of a Lie group $G$  on a manifold $M$, the Marsden-Weinstein quotient at zero of its cotangent lift is symplectomorphic to $T^*(M/G)$.
\end{remark}

\begin{proof}    
The degree $0$ constraints $\cI_0$ of $\cC$ are generated by $\{{J_0}^*w\}_{w\in \h}$, and the degree $1$ constraints $\cI_1$ are generated by  $\{{J_1}^*v\}_{v\in \g}$. Hence,
due to the  assumptions on the freeness of the $G$-action and on the regular value of $J_0$,
$\cC$ is a (smooth) graded submanifold of $\cM$. (We have $\cC=E^{\circ}[1]$ where $E=[\{{J_1}^*v\}_{v\in \g}]|_C \subset TM|_C$.)
Further  $\cC$  a coisotropic submanifold since $(J_0,J_1)$ is a Poisson map. This means that the vector fields generating the infinitesimal action $\psi$ of $\h[1] \oplus \g$ preserve $\cI$.
The graded algebra of functions on $\cC/(G\ltimes \h[1])$ is given by
$\left( C^{\infty}(\cM)/\cI \right )^{\Psi}$,  defined analogously to eq. \eqref{inva}.  
We have $$\left( C^{\infty}(\cM)/\cI \right )^{\Psi}=\cN(\cI)/\cI,$$
where $\cN(\cI):=\{F\in C^{\infty}(\cM): \{\cI,F\}\subset \cI\}$ is the Poisson normalizer of  $\cI$ in $C^{\infty}(\cM)$, hence   the quotient of $\cC$ by the action $\Psi$ coincides with the coisotropic quotient of $\cC$. 
Since  $C/E=C/G$ is smooth by our assumptions on freeness and properness,
 we can apply the first part of Prop. \ref{coisocase} to 
conclude the first half of our proof.

For the second half of the proof we notice that $\cS\in \cN(\cI)$, since the infinitesimal action $\psi$ respect the differentials $\delta$ on $\h[1] \oplus \g$
and $[X_{\cS},\cdot]$ on  ${\chi}(\cM)$ (and since the constant functions lie in $C^{\infty}_0(\cM)$). Hence $\cS$ descends to a function on 
$T^*[1](C/G)$ which commutes with itself, which by Prop. \ref{dimitri} corresponds to a Poisson bivector field on $C/G$. 
\end{proof}

\begin{remark}\label{PoisbrCG}
$C$ is a coisotropic submanifold of $(M,\pi)$ and $\sharp N^*C\subset E\subset TC$.
(This can be seen from the text preceding  Prop. \ref{coisocase} or from 
Lemmas \ref{differ} 1) and \ref{poismap} 2).)  The Poisson bracket of two functions on the quotient 
$C/G$ is computed as follows by  Lemma \ref{desc}: take their pullbacks to $C$, take any extension to $M$, apply the Poisson bracket of $M$ and restrict to $C$.
\end{remark}

\begin{remark}
Using the identification between the action $\Psi$ on $\cM=T^*[1]M$ and the action of $G \rtimes \h$ on $T^*M$ (see Remark  \ref{gmVB2}) we can 
 describe the first part of Prop. \ref{MWcm} as follows:  $E=[\{{J_1}^*v\}_{v\in \g}]|_C$ consist of tangent spaces along $C$ to the orbits of the $G$ action, and is tangent to $C$. The quotient of $E^{\circ}\subset T^*M$ is obtained dividing out $E^{\circ}$ by $\{d({J_0}^*w)|_C\}_{w\in \h}=N^*C$ and then dividing  by the cotangent lift of the $G$ action.
\end{remark}

\section{Actions on the symplectic groupoid $\Gamma$}\label{sec:groupGamma}

Until the end of this paper we assume the following set-up:  
\begin{itemize}
\item[--] $(M,\pi)$ is an integrable\footnote{This means that there exists a  \emph{symplectic groupoid} $\Gamma$ integrating the Lie algebroid $T^*M$.} Poisson manifold,
\item[--] $\h[1] \oplus \g$ is a DGLA,
\item[--]  $\psi \colon \h[1] \oplus \g \rightarrow {\chi}(\cM)$   a morphism of DGLAs with Poisson moment map\footnote{Many of the results
that follow do not require the existence of a moment map for $\psi$ at all, but just that
$\psi$ be a morphism of DGLAs with values in $\chi^{sympl}(\cM)$. We assume the existence of a Poisson moment map because it makes some results more explicit and simplifies the notation.}
 $(J_0,J_1)$. 
\end{itemize}

We denote by  $(\Gamma,\Omega)$ the source-simply connected  symplectic groupoid of the Poisson manifold $(M,\pi)$ \cite{CDW}.  In this Section we construct Lie algebra and Lie group actions on $\Gamma$, and their quotients will be the object of study of the next sections.\\

We adopt the   conventions of \cite{FOR} for the Lie groupoid $\Gamma$: the source map $\bs$ and target map $\bt$ are such that $g,h\in \Gamma$ are composable if{f} $\bs(g)=\bt(h)$. For the Lie algebroid we use the identification  
\begin{equation}\label{identi}
(Ker\; \bs_*)|_M\cong T^*M
\end{equation}
 via $v \mapsto (i_v\Omega)|_{TM}$, hence we identify sections of $T^*M$ with right invariant vector fields on $\Gamma$. Finally, the source map $\bs \colon \Gamma \rightarrow M$ is a Poisson map and the target $\bt$ is an anti-Poisson map.

\subsection{Lie algebra actions}

The infinitesimal action $\hat{\psi}$   (see Remark \ref{gmVB}) does not act by
infinitesimal Lie algebroid automorphisms of $T^*M$. Nevertheless, we can associate to it an infinitesimal action on the Lie groupoid $\Gamma$.

$\hat{\psi}$ maps  $w\in \h$ to $-d({J_0}^*w)\in  \Gamma(T^*M)$, which using \eqref{identi}   can be extended to a unique right-invariant vector field on $\Gamma$ (it is just the hamiltonian vector field  $X_{-\bt^*{J_0}^*w}^{\Gamma}$).

$\hat{\psi}$ maps $v\in \g$ to the cotangent lift of ${J_1}^*v\in \chi(M)$, which, as 
 an infinitesimal Lie algebroid automorphism of $T^*M$, by functoriality gives rise to a multiplicative vector field on $\Gamma$, which we denote by $({J_1}^*v)^{\Gamma}$.
  
Hence we obtain a \emph{linear} map $\h \times \g \rightarrow \chi(\Gamma)$, and it is natural to wonder if there is a Lie algebra structure on the domain that makes this into a Lie algebra morphism. 
To this aim recall that the DGLA $\h[1]\oplus \g$ corresponds to a 
crossed modules of Lie algebras (see Appendix \ref{A}), so in particular $\h$ is given a (non-trivial) Lie algebra structure $[\cdot,\cdot]_{\delta}$ and $\g$ acts on $\h$ by derivations of the Lie bracket. Hence we can construct the semidirect product  Lie algebra \cite[Sect. 3.14]{varLG}
$\h\rtimes \g$ by the action of $(\g,[\cdot,\cdot])$ on $(\h,[\cdot,\cdot]_{\delta})$. Explicitly the  Lie bracket on $\h\rtimes \g$ is given by $$[(w_1,v_1),(w_2,v_2)]_{\h\rtimes \g}=([\delta w_1,w_2]+[v_1, w_2]-
[v_2, w_1]\;,\; [v_1,v_2]).$$


\begin{prop}\label{infaction} If $\psi$ is a morphism of DGLAs then

\begin{center}
\fbox{\begin{Beqnarray*}
 \phi: \h\rtimes \g &\rightarrow& \chi^{sympl}(\Gamma)\\
w+v &\mapsto& X_{-\bt^*{J_0}^*w}^{\Gamma}+({J_1}^*v)^{\Gamma}
\end{Beqnarray*}}
\end{center}
is a morphism of Lie algebras. 
\end{prop} 

\begin{proof}
Since $\g$ acts on $M$ by Poissonomorphisms (Lemma \ref{differ} 2)), by functoriality it is clear that it acts on $\Gamma$ too (by symplectomorphisms), i.e., that $\phi|_{\g}$ is a Lie algebra morphism. 

We  show now that $\phi|_{\h}$ is a Lie algebra morphism. 
Since $\bt$ is an anti-Poisson map, this follows from 
\begin{equation}\label{hh}
\{{J_0}^*w_1,{J_0}^*w_2\}=X^{M}_{{J_0}^*w_1}({J_0}^*w_2)=({J_1}^*(\delta w_1)){J_0}^*w_2= {J_0}^*[\delta w_1, w_2]={J_0}^*[(w_1,0),(w_2,0)]_{\h\rtimes \g}
\end{equation}
where the second equality holds by  Lemma \ref{differ} 1) and the third equality holds by   Lemma 
\ref{poismap} 2).

We are left with showing that  
$[\phi(w),\phi(v)] =\phi[w,v]_{\h\rtimes \g}$ for $w\in \h$ and $v\in \g$.
We have
\begin{equation}\label{hg}
({J_1}^*v)^{\Gamma}(\bt^*{J_0}^*w)=\bt^*(({J_1}^*v){J_0}^*w)=\bt^*{J_0}^*[v, w]
 =-\bt^*{J_0}^*[(w,0),(0,v)]_{\h\rtimes \g},
 \end{equation}
the first equality because $({J_1}^*v)^{\Gamma}$ $\bt$-projects to ${J_1}^*v$ and the second by   Lemma 
\ref{poismap} 2).
So
$$[X_{-\bt^*{J_0}^*w}^{\Gamma},({J_1}^*v)^{\Gamma}]=X^{\Gamma}_{({J_1}^*v)^{\Gamma}(\bt^*{J_0}^*w)}=
X^{\Gamma}_{-\bt^*{J_0}^*[(w,0),(0,v)]_{\h\rtimes \g}}=\phi([(w,0),(0,v)]_{\h\rtimes \g})$$ where we used \eqref{hg} in the second equality.
\end{proof}

\subsection{Lie group actions}\label{defPhi}
Let $G$ and $H$ be the simply connected Lie groups integrating $(\g, [\cdot,\cdot])$ and $(\h, [\cdot,\cdot]_{\delta})$ respectively. The $\g$-module structure on $\h$ integrates to  a left action $\varphi$ of $G$ on $H$ by group automorphisms. Hence we can construct the \emph{semi-direct product} $H \rtimes G$, 
which is the simply connected Lie group integrating the Lie algebra $\h \rtimes \g$ defined just before Prop. \ref{infaction}  \cite[Sect. 3.15]{varLG}.
The group multiplication on $H \rtimes G$ is given by
\begin{equation}\label{semdirgr}
(h_1,g_1)\cdot(h_2,g_2)=(h_1\cdot \varphi(g_1)h_2,g_1\cdot g_2).
\end{equation}
Notice that we can decompose elements of $H \rtimes G$ as
\begin{equation} \label{semdirgr2}
(h,g)=(h,e)\cdot(e,g)=(e,g)\cdot(\varphi(g^{-1})h,e)
\end{equation}
(where $e$  denotes the identity element of the group $H$ or $G$).

Hence the infinitesimal action $\phi$ defined in Prop. \ref{infaction} -- which we assume to be complete -- integrates to a left\footnote{The action $\Phi$ satisfies 
$\frac{d}{dt}|_0\Phi(exp(tv),x)=-(\phi(v))|_x$ for all $v\in \h \rtimes \g$ and $x\in \Gamma$.}
group action of $H \rtimes G$ on $\Gamma$:
\begin{equation*}
\boxed{\Phi \colon (H \rtimes G) \times \Gamma \rightarrow \Gamma.}
\end{equation*}

The vector fields $X_{-\bt^*{J_0}^*w}^{\Gamma}$ on $\Gamma$ appearing in the infinitesimal action $\phi$ are not multiplicative vector fields. Hence the group action $\Phi$ does \emph{not} act by  groupoid automorphisms of $\Gamma$, i.e., 
$kx \circ ky$ will usually not agree with $k(x \circ y)$ for all $k\in H \rtimes G$, where $\circ$ denotes the groupoid multiplication on $\Gamma$.  
The compatibility condition between the action $\Phi$ and the multiplication $\circ$ on $\Gamma$ involves the group structure on $H \rtimes G$, as well as the group morphism
$\partial \colon H \rightarrow G$ obtained integrating $\delta \colon \h \rightarrow \g$. It is the following (see Thm. \ref{catgroupoact} for a categorical interpretation).

\begin{prop}\label{kxky} Assume that the $G$-action on $M$ is free. 
Let $x,y\in \Gamma$ be composable elements (i.e., $x \circ y$ exists) and let $k_1=(h_1,g_1),k_2=(h_2,g_2)$ be elements of $H \rtimes G$.

The elements  $k_1x$ and $k_2y$ are composable if{f} $g_1=(\partial h_2) g_2$.
In that case $$k_1x \circ k_2y= (h_1h_2,g_2)(x \circ y).$$
\end{prop}


\subsection{Proof of proposition \ref{kxky}}\label{proofkxky}


We start considering the two special cases in which either $k_1$ or $k_2$ is the identity element of $H \rtimes G$. 
\begin{lemma}\label{kxy}
Let $x,y\in \Gamma$ be composable  and let $k\in H \rtimes G$ so that
  $kx$ and $y$ are composable.  Then
  $kx \circ y= k(x \circ y)$. 
\end{lemma}
\begin{proof} Let $k=(h,g)$. 
We claim that the  composability assumptions on the pairs $(x,y)$
and $(kx,y)$ imply that $k$ is of the form $(h,e)$.
Indeed 
\begin{equation}\label{source}
\bs(kx)=\bs[(h,e)(e,g)x]=\bs[(e,g)x]=g\bs(x),
\end{equation}
where in the second equality we used  the fact that the vector fields $X_{-\bt^*{J_0}^*w}$ are tangent to the $\bs$-fibers,
 and in the last equality that $G$ acts by groupoid automorphisms of $\Gamma$. Since $\bs(kx)=\bt(y)=\bs(x)$, from the freeness of the $G$-action on $M$ at $\bs(x)$ we conclude that $g=e$, proving our claim.
 
Assume first that $h=exp(w)$ for some $w\in \h$. The diffeomorphism of $\Gamma$ induced by the action of  $k=(h,e)$ is the time-$1$ flow of  $X_{\bt^*{J_0}^*w}^{\Gamma}$.
Writing $z:=(h,e)(\bt(x))\in \Gamma$ we have
\begin{equation}\label{zx}
z\circ x=(h,e)x,
\end{equation}
since the groupoid action of $\Gamma$ on itself by \emph{left} multiplication is generated by \emph{right} invariant vector fields (such as $X_{\bt^*{J_0}^*w}^{\Gamma}$).
%
 Hence 
 \begin{equation}\label{case1}
 kx\circ y=(z\circ x)\circ y=z\circ (x\circ y)=k(x\circ y),
\end{equation}
 where the first equality uses \eqref{zx} and the last one uses \eqref{zx} applied to $x\circ y$.
 
To conclude we consider the case  $k=(h,e)$ for $h$ a general element of $H$. Write $h=h_1\cdots h_n$ where the $h_i$ are in the image of the exponential map of $\h$. The conclusion follows from $$kx \circ y=k_1x^{\prime}\circ y= k_1(x^{\prime}\circ y)=\cdots=k(x\circ y)$$
where $k_i:=(h_i,e)$ and $x^{\prime}:=k_2\cdots k_n x$. In the second equality we used $\eqref{case1}$ (notice that $x^{\prime}$ and $y$ are composable).
 \end{proof}

\begin{lemma}\label{xky}
Let $x,y\in \Gamma$ be composable  and  $k=(h,g)\in H \rtimes G$. Then
  $x$ and $ky$ are composable if{f} $k$ is of the form $(h,\partial h^{-1})$. In this case
  $x \circ ky= k(x \circ y)$. 
\end{lemma}
\begin{proof}  
We start claiming that the target map $\bt \colon \Gamma \rightarrow M$ intertwines the action of $h\in H$ on $\Gamma$ and the action of $\partial h\in G$ on $M$. 
Assume first that $h$ lies in the image of the exponential map of $\h$, i.e., that $h=\exp(w)$ for some $w\in \h$.
The action of $h$ on $\Gamma$ is obtained taking the  time-$1$ flow of the vector field  $X_{\bt^*{J_0}^*w}^{\Gamma}$.
We have $\bt_*(X_{\bt^*{J_0}^*w}^{\Gamma})=- X^{M}_{{J_0}^*w}=-J_1^*(\delta w)$ since $\bt$ is an anti-Poisson map and by Lemma \ref{differ} 1). Further
$\exp(\delta w)=\partial h$ since $\partial$ is the Lie group morphism integrating $\delta$, so the time-$1$ flow of $-J_1^*(\delta w)$ is given by the element $\partial h\in G$ under the action of $G$ on $M$, proving the claim for $h$ in the image of the exponential map. If $h$ is a general 
  element of $H$, write $h=h_1\cdots h_n$ where the $h_i$ are in the image of the exponential map of $\h$, and apply the above reasoning inductively starting from $h_1$.  

Therefore
\begin{equation}\label{target}
\bt(ky)=\bt((h,e)(e,g)y)= (\partial h) \bt((e,g)y)=(\partial h )g\bt(y)
\end{equation}
where in the last equality we used that $G$ acts by groupoid automorphisms.  The elements $x$ and $ky$ are composable if  $\bt(ky)$ agrees with $\bs(x)=\bt(y)$. Since the $G$-action on $M$ is free we conclude that this is equivalent to $(\partial h )g=e$, proving the first part of the lemma.

To show the second part of the lemma, assume first that $h=\exp(w)$ for some $w\in \h$.
On the Lie group $H \rtimes G$ we have $(h,\partial h^{-1})=(h,e)(e,\partial h^{-1})=\exp (w)\exp(-\delta w)=\exp(w - \delta w)$, where in the last equality we used the Baker-Campell-Hausdorff formula and the fact that $[(w,0),(0,\delta w)]_{\h \rtimes \g}=0$. 
Further
\begin{equation}\label{xgxsxt} 
-(\phi(w- \delta w)=
X_{\bt^*{J_0}^*w}^{\Gamma}+({J_1}^*\delta w)^{\Gamma}=
X_{\bt^*{J_0}^*w}^{\Gamma}+(X^{M}_{{J_0}^*w})^{\Gamma}=
X_{\bs^*{J_0}^*w}^{\Gamma}
\end{equation}
where we used 1) of Lemma \ref{differ} in the first equality and Lemma \ref{camille} below in the second.
Consider the path $h(t):=\exp(tw)$ in $H$ from $e$ to $h$, inducing the path  
$k(t):= (h(t), \partial h^{-1}(t))$ in $H \rtimes G$ and
\emph{two} paths in $\Gamma$: one is 
$\tau(t):=k(t)(x\circ y)$, the other $\gamma(t):=x\circ k(t)y$.
Using \eqref{xgxsxt} we see that their velocity at time $t$ is 
\begin{eqnarray}
\dot{\tau}(t)&=&(X_{\bs^*{J_0}^*w}^{\Gamma})|_{\tau(t)}\\
\dot{\gamma}(t)= (L_x)_*[(X_{\bs^*{J_0}^*w}^{\Gamma})|_{k(t)y}]&=&(X_{\bs^*{J_0}^*w}^{\Gamma})|_{\gamma(t)},
\end{eqnarray}
  where the last equality follows from the fact that $X_{\bs^*{J_0}^*w}^{\Gamma}$ is a left invariant vector field.
Hence $\gamma$ and $\tau$ are integral curves of the same vector field, and since they both start at  $(x\circ y)$ we conclude that 
\begin{equation}\label{case2}
x \circ ky=\gamma(1)=\tau(1)= k(x \circ y).
\end{equation}
Now assume that $k=(h,\partial h^{-1})$ for an arbitrary $h\in H$, and as above 
we write $h=h_1\cdots h_n$ for elements $h_i$ in the image of the exponential map.
We have $$(h,\partial h^{-1})=(h_n,\partial h_n^{-1})\cdots(h_1,\partial h_1^{-1})$$
using \eqref{xmgr},
hence the lemma is proven applying recursively\footnote{
Notice that the composability assumptions are satisfied since the inverse of
$(h_i,\partial h_i^{-1})$ is $(h_i^{-1},\partial h_i)$.} eq. \eqref{case2}.
\end{proof}

The following is a special case (with $G=\RR$) of Thm. 3.3 (ii) of \cite{FOR}.

\begin{lemma}\label{camille}
Given any function $f$ on a symplectic groupoid $\Gamma$ we have
\begin{equation}
(X^{M}_{f})^{\Gamma}=X_{\bs^*f}^{\Gamma}-X_{\bt^*f}^{\Gamma}.
\end{equation}
\end{lemma}
 \begin{proof}[Proof of Prop. \ref{kxky}]
Recall from  \eqref{semdirgr2} that any element of  $H \rtimes G$ can be written as
$(h,g)=(e,g)(\tilde{h},e)$ where $\tilde{h}:=\varphi(g^{-1})h$ (and $e$ denotes the identity of $H$ or $G$). Assume first that $k_1x$ and $k_2 y$ are composable.  We have
 \begin{eqnarray}
 k_1x &\circ& k_2 y=\\
 \label{secline2}
(e,g_1)(\tilde{h}_1,e)x &\circ& (e,g_2)(\tilde{h}_2,e)y=\\
\label{secline3}(e,g_1)\Big[(\tilde{h}_1,e)x &\circ& \un{(e,g_1^{-1}g_2)(\tilde{h}_2,e)}y\Big]=\\
\label{secline4}(e,g_1) \un{(e,g_1^{-1}g_2)(\tilde{h}_2,e)}\Big[(\tilde{h}_1,e)x &\circ& y\Big]=\\
\label{secline5} (e,g_1) \un{(e,g_1^{-1}g_2)(\tilde{h}_2,e)}(\tilde{h}_1,e)\Big[x &\circ& y\Big]
\end{eqnarray}
where in the second equality we used that the $G$ action on $\Gamma$ is by groupoid automorphisms, in the third equality we applied the second part of Lemma \ref{xky} to the underlined term $k:=(e,g_1^{-1}g_2)(\tilde{h}_2,e)$ (notice that $(\tilde{h}_1,e)x$ and $y$ are composable),
and in the fourth equality we applied Lemma \ref{kxy} to $k':=(\tilde{h}_1,e)$.

The composability assumption on $k_1x$ and $k_2 y$ is clearly  equivalent to the composability of $(\tilde{h}_1,e)x$ and $ky$ (see \eqref{secline3}).
By the first part of Lemma \ref{xky} applied to \eqref{secline3}, this is equivalent to $k$ being of the form $(h,\partial h^{-1})$ for some $h\in H$. Now $k= 
(e,g_1^{-1}g_2)(\tilde{h}_2,e)=(\varphi(g_1^{-1})h_2,g_1^{-1}g_2)$, so it follows that 
$\partial [\varphi(g_1^{-1})h_2]=g_2^{-1}g_1$, which using  Def. \ref{xmgr}
means $\partial h_2^{-1}g_1=g_2$. 
We conclude that $k_1x$ and $k_2 y$ are composable if{f} $g_1=(\partial h_2) g_2$,
 proving the first part of the proposition.

Since $g_1=(\partial h_2) g_2$ we have $$\tilde{h}_1=\varphi(g_1^{-1})h_1=
\varphi(g_2^{-1})\varphi(\partial h_2^{-1})h_1= \varphi(g_2^{-1})(h_2^{-1}h_1h_2)$$
(the last equality  using Def. \ref{xmgr}). Using this and
 $\tilde{h}_2=\varphi(g_2^{-1})h_2$
we get $\tilde{h}_2 \tilde{h}_1=\varphi(g_2^{-1})(h_1h_2)$, and 
 we can simplify the four terms before the square bracket in  \eqref{secline5} to
$$ (e,g_2)(\tilde{h}_2\tilde{h}_1,e)=
(e,g_2)\cdot(\varphi(g_2^{-1})(h_1h_2),e)=(h_1h_2,g_2),
$$ 
finishing the proof of Prop. \ref{kxky}.
\end{proof}

\section{The global quotient of $\Gamma$}\label{sec:glq}

We assume the set-up of Section \ref{sec:groupGamma}.
The action $\Phi$ of $H \rtimes G$ on $\Gamma$ (see   
Subsection \ref{defPhi})  is \emph{not} by Lie groupoid automorphisms.
In this Section we show that, in spite of this, the quotient space has an induced Lie groupoid structure. Further we show that it is a  \emph{Poisson groupoid}, i.e.,  a Lie groupoid endowed with a Poisson structure for which the graph of the multiplication is coisotropic \cite{alancoiso}.

\begin{prop}\label{groidgl} Suppose that the $H \rtimes G$-action on $\Gamma$ is free and proper. Then
there is a   Poisson groupoid structure on 
 \[
\xymatrix{\Gamma/(H \rtimes G)  \ar[d]\ar[d]\ar@<-1ex>[d]\\ M/G}
\]
for which the natural projection $p \colon \Gamma \rightarrow \Gamma/(H \rtimes G)$ is a groupoid morphism and a Poisson map. 
\end{prop}
\begin{proof} We first show that there is an induced groupoid structure on the quotient
$\Gamma/(H \rtimes G)$. We have
$\bs[(H \rtimes G)x]\subset G\bs(x)$ and $\bt[(H \rtimes G)x]\subset G\bt(x)$ for each $x\in \Gamma$. Indeed by Prop.  \ref{infaction} the infinitesimal $\h \rtimes \g$-action  $\phi$ on $\Gamma$ is given by
$X_{-\bt^*{J_0}^*w}^{\Gamma}+({J_1}^*v)^{\Gamma}$, which under $\bs_*$ maps to 
${J_1}^*v$ and under $\bt^*$ maps to 
$X^{M}_{{J_0}^*w}+{J_1}^*v={J_1}^*(\delta w +v)$ (the last equality by Lemma \ref{differ} 1)). Hence
there are well-defined source and target maps of $\Gamma/(H \rtimes G)$   induced from those of $\Gamma$.  

We define the multiplication of composable elements $\un{x},\un{y}\in \Gamma/(H \rtimes G)$  as $\un{x}\circ \un{y}:=p(x\circ y)$, where $x,y\in \Gamma$ are  composable elements with $p(x)=\un{x}, p(y)=\un{y}$.
Since the fibers of $p$ are exactly the $H \rtimes G$ orbits, by Prop. \ref{kxky}
the product  $\un{x}\circ \un{y}$ is independent of the choice of the choice of 
 \emph{composable} lifts $x$ and $y$. The existence of the composable  lifts $x,y$   is clear: if $\tilde{x},y$ are arbitrary  lifts to $\Gamma$ of $\un{x},\un{y}$, the composability of $\un{x}$ and $\un{y}$ implies that there exists $g\in G$ such that $\bt(y)=g\cdot \bs(\tilde{x})=\bs(g\cdot \tilde{x})$, so just take $x:=g\cdot \tilde{x}$.

There is an induced Poisson structure on  $\Gamma/(H \rtimes G)$ for which the projection $p$ is a Poisson map: the infinitesimal action $\phi$ on $\Gamma$ is given by vector fields $X_{-\bt^*{J_0}^*w}^{\Gamma}$ and $({J_1}^*v)^{\Gamma}$, which are symplectic vector fields (indeed hamiltonian vector fields, see \eqref{Jiv}.)

At last we prove the compatibility of the groupoid and Poisson structure on the quotient. 
The fact that $\Gamma$ is a symplectic groupoid means that $(graph(m_{\Gamma}))\subset \Gamma\times\Gamma\times\bar{\Gamma}$ is a lagrangian (in particular coisotropic) submanifold,
where $m_{\Gamma}$ is the multiplication of $\Gamma$ and $\bar{\Gamma}$ denotes the groupoid $\Gamma$ endowed with the negative of its symplectic form. 
The projection $p\times  p\times p$ is  a Poisson map  and maps $graph (m_{\Gamma})$ surjectively onto $graph  (m_{\Gamma/(H \rtimes G)})$. Now if $f_1,f_2$ are functions vanishing on  $graph (m_{\Gamma/(H \rtimes G)})$ then the pullbacks $(p\times  p\times p)^{-1}f_i$ to $ \Gamma\times\Gamma\times\bar{\Gamma}$
are functions which vanish on $graph (m_{\Gamma})$. Therefore their Poisson bracket also vanishes on $graph (m_{\Gamma/(H \rtimes G)})$, and hence $\{f_1,f_2\}$ vanishes on 
$graph (m_{\Gamma/(H \rtimes G)})$. This shows that  
$graph (m_{\Gamma/(H \rtimes G)})$ is a coisotropic submanifold of the corresponding triple product, i.e., that  $ {\Gamma/(H \rtimes G)  \rightrightarrows M/G}$
is a Poisson groupoid.
\end{proof}
 
We compare the global quotients of $\Gamma$ (a Poisson groupoid) with 
the global quotient of $T^*[1]M$ (which corresponds to a Lie bialgebroid). Recall that 
the infinitesimal counterpart of a Poisson groupoid is a Lie bialgebroid \cite{MK2}. 
\begin{prop}
The Lie bialgebroid structure on   $D/G$ induced by the global quotient of $T^*[1]M$ (as in Prop. \ref{glocM}) is exactly the infinitesimal counterpart of the the global quotient of  $\Gamma$ (as in Prop. \ref{gloGamma}). 
\end{prop}
 \begin{proof}
Recall from Remark \ref{bial} that a Lie algebroid structure on a vector bundle $A$ is equivalent to a homological vector field on $A[1]$, via the derived bracket construction. 
The Lie algebroid structure on $D^*/G$  induced by the global quotient of $T^*[1]M$ is given using the pushforward of the vector field $X_{\cS}$ to $(D^*/G)[1]$. Hence it is the unique one so that $T^*M\rightarrow (T^*M/D^{\circ})/G=D^*/G$ is a Lie algebroid morphism. On the other hand the groupoid structure on $\Gamma/(H \rtimes G)$ is determined by the fact that  $p \colon \Gamma \rightarrow \Gamma/(H \rtimes G)$ is a Lie groupoid morphism, which in turn  induces a morphism of Lie algebroids $p_* \colon ker(\bs_*)|_M \rightarrow 
ker(\un{\bs}_*)|_{M/G}$ where $\un{\bs}$ denotes the source map of the quotient groupoid. Under the  identification at the beginning of Section \ref{sec:groupGamma} this map is just the quotient map $T^*M\rightarrow (T^*M/D^{\circ})/G=D^*/G$. We conclude that the Lie algebroid structures on $D^*/G$ determined by Prop. 
 \ref{glocM} and  Prop. \ref{gloGamma} agree.

Recall that a Lie algebroid structure on a vector bundle $A$ is also determined by a degree -1 Poisson bracket on $A^*[1]$, or equivalently by a fiber-wise linear Poisson structure on $A^*$.
The canonical Lie algebroid structure on $TM$ is determined by the Poisson bracket corresponding to the standard symplectic form on $T^*[1]M$. The Lie algebroid structure on $D/G$   induced by the global quotient of $T^*[1]M$ is determined by 
the  projection $T^*[1]M \rightarrow (D^*/G)[1]$ being a Poisson map. On the other hand 
the projection $\Gamma \rightarrow \Gamma/(H \rtimes G)$ is a Poisson groupoid morphism
by Prop. \ref{groidgl}, so the induced map on Lie algebroids
 $T^*M \rightarrow D^*/G$ is also a Poisson map \cite[Prop. 11.4.13]{MK2}. This shows that the two Lie algebroid structures on $D/G$ determined by Prop. 
 \ref{glocM} and  Prop. \ref{gloGamma} agree.
\end{proof}

\section{Hamiltonian  reduction of $\Gamma$}\label{Sec:MWGamma}

We assume again the set-up of Section \ref{sec:groupGamma} and consider the action $\Phi$ of $H \rtimes G$ on $\Gamma$ (see   
Subsection \ref{defPhi}). In this Section  
we show that the Marsden--Weinstein quotient of $\Gamma$ at zero is a symplectic groupoid (not necessarily source simply connected) of $J_0^{-1}(0)/G$.
This generalizes results of \cite{FOR}, in  which a conventional group $G$ is acting 
(i.e., $\h=0$).

The $G$-action $J^*_1$ on $M$ is   by Poisson diffeomorphism (Lemma \ref{differ} 2)).
By Thm. 3.3 of \cite{FOR} the lifted $G$-action on $\Gamma$ is hamiltonian, and there is a canonical moment map  $J^{\Gamma}_1:\Gamma \rightarrow \g^*$ which is $G$-equivariant. In particular we have
\begin{equation}\label{Jiv}
(J_1^*v)^{\Gamma}=X^{\Gamma}_{(J^{\Gamma}_1)^*v}
\end{equation}
for all $v\in \g$. Therefore a moment map for the $H \rtimes G$-action $\Phi$ is
$$J^{\Gamma}:=(-J_0\circ \bt,J^{\Gamma}_1) \colon \Gamma\rightarrow (\h \rtimes \g)^*,$$ with components
$-\bt^*(J_0^*w)+(J^{\Gamma}_1)^*v$ where $w+v\in \h \rtimes \g$.

The following is a refinement of  Prop. \ref{infaction}.
\begin{lemma}\label{JGpois}
The moment map $J^{\Gamma}$ is a Poisson map.
\end{lemma}
\begin{proof}
$(J^{\Gamma})^*|_{\g \times \g}=(J^{\Gamma}_1)^*$ preserves Poisson brackets because, by Thm. 3.3 of \cite{FOR}, $J^{\Gamma}_1:\Gamma \rightarrow \g^*$   is a $G$-equivariant moment map. 

The computation to show that $(J^{\Gamma})^*|_{\h \times \h}$ respects brackets is \eqref{hh}.

To show that $(J^{\Gamma})^*|_{\h \times \g}$ respects brackets we compute, for $w\in \h$ and $v\in \g$, 
\begin{equation*}\{-\bt^*{J_0}^*w,(J^{\Gamma}_1)^*v\}=(X^{\Gamma}_{(J^{\Gamma}_1)^*v})(\bt^*{J_0}^*w)=-\bt^*{J_0}^*[(w,0),(0,v)]_{\h\rtimes \g}\end{equation*}
where we used \eqref{Jiv} and \eqref{hg} in the second equality.
\end{proof}

\begin{lemma}\label{JGsubgr} Assume that $0$ is a regular value of $J_0 \colon M \rightarrow \h^*$ and that the $G$-action on $C:=(J_0)^{-1}(0)$ is almost free. Then  
the identity connected component of   $(J^{\Gamma})^{-1}(0)$ is a Lie subgroupoid of 
$\Gamma$ with base $C$.
\end{lemma}
\begin{proof}
We first argue that $(J^{\Gamma})^{-1}(0)$ is a smooth submanifold of $\Gamma$.
Since $0$ is a regular value of $J_0$ it follows that the assignment $w\in \h \mapsto (X^{\Gamma}_{-\bt^*(J_0^*w)})_x$ is injective for every $x\in (J_0 \circ \bt)^{-1}(0)$.
The almost freeness assumption means that the assignment $v\in \g \mapsto (J^*_1v)_m$ is injective for every $m\in J^{-1}(0)$, hence the assignment $v\in \g \mapsto (J^*_1v)^{\Gamma}_x$  is injective for every $x\in (J_0 \circ \bt)^{-1}(0)$ (recall that  
$(J^*_1v)^{\Gamma}$ $\bt$-projects to $J^*_1v$). Since the $X^{\Gamma}_{-\bt^*(J_0^*w)}$ are tangent to the $\bs$-fibers while the $(J^*_1v)^{\Gamma}=X^{\Gamma}_{(J^{\Gamma}_1)^*v}$ are transversal to them, we conclude that  $J^{\Gamma}$ is a submersion at every point of $(J_0 \circ \bt)^{-1}(0)$. In particular $0\in (\h \rtimes \g)^*$ is a regular value of $J^{\Gamma}=(-J_0\circ \bt,J^{\Gamma}_1)$, proving our claim.

Denote by $\Sigma$ the identity connected component of   $(J^{\Gamma}_1)^{-1}(0)$,
which is a Lie subgroupoid of $\Gamma$ by Cor. 4.3 of \cite{FOR}.
We claim  that the identity connected component of   $(J^{\Gamma})^{-1}(0)$, which is  
$\Sigma \cap \bt^{-1}(C)$, is actually contained in 
$\bs^{-1}(C)$. Therefore it is the intersection of the subgroupoids $\Sigma$ and 
$\bt^{-1}(C) \cap \bs^{-1}(C)$, and therefore it is itself a subgroupoid. 
 
To prove the claim\footnote{Our proof shows that the claim holds even if we replace $C=J_0^{-1}(0)$ 
by any other level set of $J_0$.} we use that the Lie algebroid $A$ of $\Sigma$ at $p\in M$, under the identification \eqref{identi}, 
 is given by
$A_p=T_p(G\cdot p)^{\circ}$ (by Cor. 4.3 of \cite{FOR} together with eq. (13) there). We have $\sharp A|_C\subset TC$:
indeed if $p\in C$ and $\xi\in A_p=T_p(G\cdot p)^{\circ}$, then for all $w\in \h$ we have
$$\langle \sharp \xi, d(J_0^*w) \rangle =\langle   \xi, X^M_{J_0^*w} \rangle=
\langle   \xi, J_1^*(\delta w)\rangle=0
$$
where in the second equality we used Lemma \ref{differ} 1) and in the last one that the infinitesimal $G$-action on $M$ is given by the vector fields $J_1^*(v)$ for $v\in \g$. 
The subgroupoid $\Sigma$ is obtained considering  the right-invariant distribution $R_*(A)\subset ker(\bs_*)$ on $\Gamma$ and taking its leaves through points of the base $M$ (Section 3 of \cite{MMRC}).
Take a point in the identity connected component of  $(J^{\Gamma})^{-1}(0)$, i.e., take 
$g$ of $\Sigma$ with $\bt(g)\in C$. We can join $g$ to $\bs(g)$  by a path $\gamma$ in the $\bs$-fiber, so  $\bt_*(\dot{\gamma}(t))\in TC$ for all times $t$ by the above, showing that the whole path $\bt \circ \gamma$ lies in $C$. In particular $\bs(g)\in C$, proving the claim.
\end{proof}

\begin{remark} In \cite[Sect. 4.1]{FOR} a statement analogous to Lemma \ref{JGsubgr} is proved using the fact that the  moment map is a groupoid $1$-cocycle. In our case
$J^{\Gamma}=(-J_0\circ \bt,J^{\Gamma}_1)$ is  not a groupoid $1$-cocycle, unless $J_0\equiv 0$.
\end{remark}

\begin{prop}\label{gloGamma}
Assume that $0$ is a regular value of $J_0$. Assume that the $G$-action on $(J_0)^{-1}(0)$ and the $H$-action on  $[(J^{\Gamma})^{-1}(0)]_{id}$
are free, where the latter denotes the identity connected component of $(J^{\Gamma})^{-1}(0) \subset \Gamma$, and the $H \rtimes G$ action on $[(J^{\Gamma})^{-1}(0)]_{id}$ is proper. Then there are induced symplectic and Lie groupoid structures making
\[
\xymatrix{[(J^{\Gamma})^{-1}(0)]_{id}/( H\rtimes G)
  \ar[d]\ar[d]\ar@<-1ex>[d]\\ (J_0)^{-1}(0)/G}
\]
into a  symplectic  groupoid.
\end{prop}

\begin{remark}
When $\h=\{0\}$, Prop. \ref{gloGamma} recovers a symplectic groupoid of $M/G$ (see Ex. \ref{h0}).
\end{remark}

\begin{proof}
$[(J^{\Gamma})^{-1}(0)]_{id}$ is smooth by  Lemma \ref{JGsubgr}. The lifted $G$-action   and the $H$-action on it are free. We deduce that the $H\rtimes G$-action on $[(J^{\Gamma})^{-1}(0)]_{id}$ is free. Indeed, assume that $(g,h)\in H\rtimes G$  fixes $x\in \Gamma$. Then we have $\bs((e,g)x)=\bs((h,g)x)=\bs(x)$, from which we conclude that $g=e$ and hence $h=e$, showing freeness. By the properness assumption, the quotient $[(J^{\Gamma})^{-1}(0)]_{id}/( H\rtimes G)$ is smooth.

 Lemma \ref{JGpois} ensures that $J^{\Gamma}$ is $( H\rtimes G) $-equivariant, so the Marsden--Weinstein reduced space at  zero, 
  $[(J^{\Gamma})^{-1}(0)]_{id}/( H\rtimes G)$, is a  symplectic manifold. 

$[(J^{\Gamma})^{-1}(0)]_{id}$ is a Lie subgroupoid   of $\Gamma$ by Lemma \ref{JGsubgr}, and $[(J^{\Gamma})^{-1}(0)]_{id}/( H\rtimes G)$ has an induced Lie groupoid structure (see the first half of the proof of Thm. \ref{groidgl}).

To show the compatibility between the symplectic and groupoid structure, notice that the dimension of $(J^{\Gamma})^{-1}(0)/( H\rtimes G)$ is double the dimension of the base $(J_0)^{-1}(0)/G$, and that the reduced symplectic form is multiplicative because 
the symplectic form on $\Gamma$ is multiplicative.  
\end{proof}

We conclude this section comparing the structures on the Marsden--Weinstein quotients of $\Gamma$ and $T^*[1]M$.

\begin{prop}  The Poisson structure on  $(J_0)^{-1}(0)/G$ given\footnote{By the requirement that its source map $\un{\bs}$ be Poisson.} by the quotient symplectic groupoid of Prop. \ref{gloGamma} agrees with the one obtained in Prop. \ref{MWcm}.
\end{prop}
\begin{proof} Denote $C:= (J_0)^{-1}(0)$.
 Endow  $C/G$ with the Poisson structure $\{\cdot,\cdot\}_{C/G}$
obtained in Prop. \ref{MWcm}, which by Remark \ref{PoisbrCG} is computed pulling back functions on $C/G$ to functions on $C$.
  Consider  the commutative diagram:
\[
\xymatrix{
\Gamma \;\; \supset & (J^{\Gamma})^{-1}(0)\ar[r]^{\tilde{\pi}}\ar[d]^{\bs}& (J^{\Gamma})^{-1}(0)/(H \rtimes G) \ar[d]^{\un{\bs}}\\
M \;\; \supset & C\ar[r]^{\pi}&C /G
}.
\]
Given functions $f_1,f_2$ on $C/G$ we have
$$ (\un{\bs}\circ \tilde{\pi})^*\{f_1,f_2\}_{C/G}=(\pi \circ \bs)^*\{f_1,f_2\}_{C/G}=\{(\pi \circ \bs)^*f_1,(\pi \circ \bs)^*f_2\} 
= \tilde{\pi}^*\{ \un{\bs}^*f_1, \un{\bs}^*f_2\}.$$ 
Here we used that $\pi \circ \bs$ and $\tilde{\pi}$ preserve Poisson brackets 
(upon extensions of functions, which cause no problems since the two inclusions on the left are inclusions of  coisotropic submanifolds). Now just use the fact that $\tilde{\pi}^*$ is injective, to conclude that $\un{\bs}$ is a Poisson map.
\end{proof}

\section{Lie 2-group actions}\label{sec:cga}

In this Section we show that the action $\Phi$ of $H\rtimes G$ on $\Gamma$ defined in 
Subsection \ref{defPhi} is an action in \textbf{Gpd}, the category of Lie groupoids (rather than just an action in the category of smooth manifolds). 

We start recalling some notions from  Appendix \ref{A}.
The DGLA $\h[1]\oplus \g$ corresponds to a crossed module of Lie algebras, which integrates to a crossed module of Lie groups.
A crossed module of Lie groups gives rise to a (strict) Lie 2-group, as follows.
 Given a crossed module of Lie groups
$(H,G,\partial, \varphi)$, consider  the  action of $H$ on $G$  by $h\mapsto \partial h \cdot$ and the induced transformation groupoid  
\[
\xymatrix{H\times G \ar[d]\ar[d]\ar@<-1ex>[d]\\ G},
\]
which we denote by $\cG$.
This means that the source and target maps are given by $$\bs(h,g)=g\;\;\;\;\;\;\;\;\;\;\;\;\bt(h,g)=\partial h \cdot g$$ and the groupoid multiplication is $(h_1,g_1)\circ (h_2,g_2)=(h_1h_2,g_2)$. Endow the spaces of arrows with the semi-direct group structure given by  the action $\varphi$ of $G$ on $H$ (see \eqref{semdirgr}). The space of objects  $G$ already has a group structure. These data make $\cG$ into a Lie-2 group (Def. \ref{catgr}).
 
Thinking of a Lie 2-group as a   group object in \textbf{Gpd}
leads to define a \emph{strict action} of a Lie 2-group $\cG$ on 
a Lie groupoid $\Gamma$ as a Lie groupoid morphism $\cG\times \Gamma \rightarrow \Gamma$ which is also a group action (for the groups
$H \rtimes G$ and $G$).
    
\begin{thm}\label{catgroupoact}
The group actions $\Phi$ of $H \rtimes G$ on $\Gamma$ (integrating $\phi$) and  of $G$ on $M$ (integrating $J^*_1$, which we assume to be a free action) combine to a strict action of the Lie 2-group $H \times G {\rightrightarrows} G$ on
$\Gamma{\rightrightarrows} M$.  
\end{thm} 
\begin{proof}
We have to check that the map 
\[
\Phi \colon \xymatrix{
(H\times G) \ar[d]\ar[d]\ar@<-1ex>[d]   \times \Gamma 
 \ar[r] & \Gamma  \ar[d]\ar[d]\ar@<-1ex>[d]
\\ 
G   \times   M \ar[r] & M  
}
\]
is a groupoid morphism, i.e., that it respects source map, target map and multiplication.

Let $(h,g) \in H\rtimes G$ and $x\in \Gamma$. The source is preserved because $\bs((h,g)x)=g\bs(x)$ by \eqref{source}. 
The target is preserved because because
$\bt((h,g)x)=(\partial h) g\bt(x)$ by \eqref{target}.

To show that the multiplication is preserved let
$(h_1,g_1,x_1)$ and  $(h_2,g_2,x_2)$ be composable elements of the product  groupoid
$(H \times G)\times \Gamma$. This means that $g_1=(\partial h_2) g_2$ (by the definition of the transformation groupoid $H \times G {\rightrightarrows} G$) 
 and that $x_1$ and $x_2$ are composable. We have to show that applying first $\Phi$  and then the multiplication in $\Gamma$ yields the same as applying first the multiplication in the groupoid $(H \times G)\times \Gamma$
  and then  $\Phi$. In other words, we have to show
\begin{equation}\label{catgract}
(h_1,g_1)x_1\circ (h_2,g_2)x_2= (h_1h_2,g_2)[x_1\circ x_2].\end{equation} This equation holds by Prop. \ref{kxky}.
\end{proof}   


Wockel defines in  \cite[Def. I.8]{Wo} the notion of \emph{principal $\cG$-2-bundle} over a base manifold $N$, where $\cG$ is a (strict) Lie 2-group.
 It is a categorified version of the usual notion of principal bundle.
One reason why  principal $\cG$-2-bundles are  interesting is the following \cite[Rem. II.11]{Wo}: 
when the Lie 2-group $\cG$ corresponds to a crossed module of Lie groups of the form $(H,Aut(H),\partial, \varphi)$, where $H$ is a Lie group and $\partial \colon H \to Aut(H)$ is given by conjugation, then  principal $\cG$-2-bundles  define gerbes over $N$.

We show that, under certain assumptions, our action $\Phi$ defines a principal $\cG$-2-bundle:
\begin{prop}
If the action $\Phi$ is such that both the $G$-action on $M$ and the 
$H \rtimes G$-action on $\Gamma$ are free and proper with the same quotient $N:=M/G=\Gamma/(H \rtimes G)$, then the action $\Phi$ makes $$\Gamma \overset{\pi}{\to}  N$$ into a principal $\cG$-2-bundle.
\end{prop}
\begin{proof}
We have to check one by one the items of Definition I.8 of \cite{Wo}.
 \begin{enumerate}
\item[a)] $\Phi$ is both a Lie group action and a morphism of Lie groupoids by Thm. \ref{catgroupoact}. 

\item[b)] The projection $\pi \colon \Gamma \to \Gamma/\cG=N$ is a Lie groupoid morphism by Prop. \ref{groidgl}. 

\item[c)] If $\{U_i\}$ is an open cover of $N$ over which both the $G$-bundle $M$ and the $H \rtimes G$-bundle $\Gamma$ are trivial, then there exist  Lie groupoid isomorphisms
$\Gamma|_{\pi|_M^{-1}(U_i)}\cong U_i\times \cG$ which intertwine the $\cG$-action by $\Phi$ on the L.H.S. and the $\cG$-action by left multiplication on the R.H.S.

To prove this we first need\\
\noindent \emph{Claim: let $x\in U_i$ and $\hat{x}\in M$ with $\pi(\hat{x})=x$. Then
\begin{equation}\label{groidiso}
\cG \cong \Phi(\cG,\hat{x}),\;\;\;\; k \mapsto \Phi (k, \hat{x})
\end{equation}
is a Lie groupoid isomorphism (where the R.H.S. is seen as a subgroupoid of $\Gamma$)  and intertwines the $\cG$-actions.}

\noindent To prove the claim notice that  the above map is a bijection because the $\cG$-action on $\Gamma$ is free.
It is a Lie groupoid morphism as consequence of the fact that  $\Phi$ is a Lie groupoid morphism (Thm. \ref{catgroupoact}). It intertwines the $\cG$-actions because 
$\Phi$ is a Lie group action.

By the principality assumptions we can choose a section of $\pi|_M \colon M \rightarrow N$ over $U_i$.  For every $x\in U$ denote by $\hat{x}\in M$ its image under the section.  Since 
$\Gamma/\cG$ is the trivial groupoid $N {\rightrightarrows}N$ it follows that 
the orbits of the Lie groupoid $\Gamma {\rightrightarrows}M$ lie inside the fibers of $\pi|_M \colon M \rightarrow N$. Hence we can write $\Gamma|_{\pi|_M^{-1}(U_i)}=\cup_{x\in U_i}\Gamma|_{\pi|_M^{-1}(x)}=\cup_{x\in U_i}  \Phi(\cG,\hat{x})$, and 
 from \eqref{groidiso} we conclude that
\begin{equation}
\varphi_i \colon  U_i\times \cG \cong \Gamma|_{\pi|_M^{-1}(U_i)},\;\;\;\;\; (x,k)\mapsto \Phi (k, \hat{x})
\end{equation}
is an isomorphism of Lie groupoids intertwining the $\cG$-actions.

\item[d)] $\varphi_i$ intertwines the first projection onto $U_i$ and $\pi$ since $\pi(\hat{x})=x$.
\end{enumerate}

\end{proof}

\section{Examples}\label{secex}
\addtocontents{toc}{\protect\mbox{}\protect}

Let $(M,\pi)$ be a Poisson manifold and $\Gamma$ its source simply connected symplectic groupoid. 

We present   examples for the global quotient and the Marsden--Weinstein quotient at zero of $\Gamma$ (the Poisson  groupoid of Prop. \ref{groidgl}
 and  the symplectic groupoid of Prop. \ref{gloGamma}   respectively).

In the first example we let $\h=0$ and recover a symplectic groupoid for $M/G$, as in \cite{FOR}.
\begin{ep}\label{h0}
Let $\g$ be a Lie algebra. View it as a DGLA concentrated in degree $0$. This corresponds to the  crossed module of Lie algebras (Def. \ref{xmla}) obtained setting $\h=0$.

A morphism of DGLAs $\g \rightarrow \chi(T^*[1]M)$ with Poisson moment map corresponds
simply to an action of $\g$ on $M$ by Poisson vector fields (see Cor. \ref{summ}). 
In this case the action $\Phi$  on $\Gamma$ (Section \ref{defPhi}) is just the lift of the $G$-action on $M$. Hence
the Marsden--Weinstein quotient of Prop. \ref{gloGamma} is the symplectic groupoid for $M/G$ constructed in Cor. 4.7 of \cite{FOR} (which generally is not source-simply connected).
\end{ep}

In the second example we let $\h=\g$, recover hamiltonian actions,  and construct a symplectic groupoid for the ``Marsden--Weinstein quotient'' of $M$.
 
\begin{ep}\label{RuiDavid}
Let $\g$ be a Lie algebra. Then $\delta=Id \colon \g \rightarrow \g$, together with the adjoint action, gives a crossed module of Lie algebras. Denote  the corresponding DGLA by $\g[1]\oplus \g$. 
A morphism of DGLAs $\psi \colon \g[1]\oplus \g \rightarrow \chi(\cM)$ with Poisson moment map corresponds
 exactly to a  $\g$-action on $(M,\pi)$ with equivariant moment map $J_0 \colon M\rightarrow \g^*$ (see Cor. \ref{summ}). The corresponding infinitesimal action on the symplectic groupoid $\Gamma$ is
\begin{eqnarray*}
 \phi: \g\rtimes \g &\rightarrow& \chi^{sympl}(\Gamma)\\
w+v &\mapsto& X_{-\bt^*{J_0}^*w}^{\Gamma}+X^{\Gamma}_{\bs^*{J_0}^*v-\bt^*{J_0}^*v},
\end{eqnarray*} 
by Prop. \ref{infaction} together with Lemma \ref{camille}. This integrates to a Lie group action $\Phi$ of $G\rtimes G$ on $\Gamma$, whose
 Marsden--Weinstein quotient  as in Prop. \ref{gloGamma} is a symplectic groupoid for $J_0^{-1}(0)/G$, the Marsden--Weinstein quotient of the $G$ action on $M$.

The global quotient of  $\Gamma$ as in Prop. \ref{groidgl} 
is a Poisson groupoid over $M/G$ of dimension $2\dim(M/G)$. However it is \emph{not} a symplectic groupoid in general. For instance consider the $\g=\RR$ action on $M=(\RR^4, \pd{y_1}\wedge \pd{y_2}+\pd{y_3}\wedge \pd{y_4})$ with moment map $J_0=y_4$. The action $\Phi$ of $G\rtimes G$ on $\Gamma=\RR^4\times \RR^4$ has moment map $(-x_4,y_4-x_4)$, and the global quotient of $\Gamma$ is the pair groupoid $\RR^3\times \RR^3$ with Poisson structure 
$-\pd{x_1}\wedge \pd{x_2}+\pd{y_1}\wedge \pd{y_2}.$
\end{ep}

\begin{remark}
 Thm. 3.6 of \cite{FI} considers   Poisson actions of  Poisson Lie groups on $(M,\pi)$ with moment map $J_0$. In the special case in which the Poisson Lie group $G$ has trivial Poisson structure, their construction is the following:
they lift the action to a hamiltonian action of the \emph{product} group $G\times G$  on $\Gamma$,  and state that the Marsden--Weinstein quotient is a symplectic groupoid of $J_0^{-1}(0)/G$.

The corresponding  infinitesimal action and our action $\phi$  in Ex. \ref{RuiDavid}  coincide\footnote{Up to a global sign, due to different  conventions.} by means of the   Lie algebra isomorphism $\g \rtimes \g \cong \g \times \g,\;(w,v)\mapsto (w+v,v)$, so this $G \times G$ action and the global action $\Phi$ in Ex. \ref{RuiDavid}  coincide by means of the
Lie group isomorphism 
$G\rtimes G\cong G\times G,  \;(h,g)\mapsto (hg,g)$.
\end{remark}

In the third example we let $\h$ be an arbitrary $\g$-module, generalizing Ex. \ref{h0}.
\begin{ep}
Let $\g$ be a Lie algebra and $\h$ a $\g$-module. Viewing $\h$ as
an abelian Lie algebra and setting $\delta \colon \h \rightarrow \g$ to be the zero map we obtain a crossed module.
A morphism of DGLAs $\h[1]\oplus \g \rightarrow \chi(T^*[1]M)$ with Poisson moment map corresponds
to an action of $\g$ on $M$ by Poisson vector fields and a $\g$-equivariant map $J_0 \colon M\rightarrow \h^*$ (which we assume to be a submersion) whose fibers are Poisson submanifolds of $M$  (see Cor. \ref{summ}). An instance of such a situation is given by $M=N\times \h^*$, where $N$ is a Poisson manifold on which $\g$ acts by Poisson vector fields: as $\g$ action on $M=N\times \h^*$ one can take the diagonal action, and as $J_0$ just the natural projection. The quotient $M/\g$ is interesting because usually it is not a product (even though $M=N\times \h^*$ is).
\end{ep}
 
 \appendix
\section{DGLAs, crossed modules, Lie 2-groups}\label{A}

In this Appendix we explain the correspondences between the following algebraic structures:
\begin{itemize}
\item  DGLAs concentrated in degrees $-1$ and $0$
\item crossed modules of Lie algebras
\item crossed modules of Lie groups
\item Lie 2-groups.
\end{itemize}


\begin{defi}\label{dla}

\noindent
  A \emph{graded Lie algebra} consists of a a graded vector space\footnote{We assume that $L$ is bounded, i.e., that there exists an integer $I$ such that $L_i=\{0\}$ whenever $|i|>I$.} $L=\oplus_{i\in \ZZ} L_i$
together with a  bilinear bracket $[\cdot,\cdot] \colon L \times L \rightarrow L$ such that
\begin{itemize}
\item[--] the bracket is degree-preserving: $[L_i,L_j]\subset L_{i+j}$ 
 \item[--] the bracket is graded skew-symmetric:
 $[a,b]=-(-1)^{|a||b|}[b,a]$ 
\item[--] the adjoint action $[a,\cdot]$ 
is a degree $|a|$ derivation of the bracket: $[a,[b,c]]=[[a,b],c]+(-1)^{|a||b|}[b,[a,c]]$.
\end{itemize}
  A \emph{differential graded Lie algebra} (DGLA) $(L,[\cdot,\cdot], \delta)$ is a graded Lie algebra together with a linear $\delta : L \rightarrow L$ such that
\begin{itemize}
 \item[--] $\delta$ is a degree $1$ derivation of the bracket:
$\delta(L_i)\subset L_{i+1}$ and $\delta[a,b]=[\delta a,b]+ (-1)^{|a|}[a, \delta b]$ 
\item[--]  $\delta^2=0$.
\end{itemize}
Above $a,b,c$ are homogeneous elements of $L$ of degrees $|a|,|b|,|c|$ respectively.
 
\end{defi}

\begin{defi}\label{xmla} {\cite{GerstRIngAlg}}
\noindent
A \emph{crossed module of Lie algebras} (or \emph{differential crossed module}) consists of Lie algebras $\h$, $\g$ with a Lie algebra morphism $\delta \colon \h \rightarrow \g$ and a left Lie algebra action $\lambda$  of $\g$ on $\h$ by derivations satisfying
\begin{itemize}
 \item[--]    $\lambda(\delta w)=[w, \cdot]_{\h}$
\item[--] $\delta(\lambda(v)w)=[v, \delta w]_{\g}$  
\end{itemize}
 for all $w\in \h$ and $v\in \g$.
\end{defi}

\begin{lemma}\label{dglaxm}
DGLAs concentrated in degrees $-1$ and $0$ are in one-to-one correspondence with crossed modules of Lie algebras. 
\end{lemma}

This lemma is a well-known fact. An explicit proof
is given in \cite{ZZL}; here we recall one direction of this correspondence.  
A DGLA concentrated in degrees $-1$ and $0$ is  of the form $L=\h[1]\oplus \g$ for usual vector spaces $\h,\g$.
To such DGLA
$(\h[1] \oplus \g, [\cdot,\cdot],\delta)$ one associates 
\begin{itemize}
 
 \item[--] the Lie algebra $(\h,[\cdot,\cdot]_{\delta})$ where $[w_1,w_2]_{\delta}=[\delta w_1,w_2]$
\item[--] the Lie algebra $\g$
\item[--] the Lie algebra morphism $\delta \colon \h \rightarrow \g$
\item[--]  the action   of $\g$ on $\h$ given by $v \mapsto [v,\cdot]$.
\end{itemize}


\begin{defi}\label{xmgr}
{\cite{MR1170713}\cite{MR1001474}}
A \emph{crossed module of Lie groups} consists of groups $H,G$, a homomorphism $\partial \colon H \rightarrow G$ and a left action $\varphi$ of $G$ on $H$ by group automorphisms such that for all $h\in H$ and $g\in G$:
\begin{itemize}
\item[--] $\varphi (\partial h_1)h_2=h_1 h_2 h_1^{-1} $
\item[--] $\partial(\varphi(g)h)=g \partial h g^{-1}$.
\end{itemize}
\end{defi}

Crossed modules of Lie algebras clearly integrate to crossed modules of Lie groups for which the groups $G$ and $H$ are simply connected.

\begin{defi}\label{catgr}\cite[Sec. 3]{FB}
A \emph{(strict) Lie 2-group} (also known as \emph{categorical group}) 
is a group object in \textbf{Gpd}, where \textbf{Gpd} denotes the category of Lie groupoids and (strict) Lie groupoid homomorphisms. In other words, a Lie 2-group
consists of a Lie groupoid $\cG$ and (strict) Lie groupoid morphisms $\circ: \cG\times \cG \rightarrow \cG$, $e:  (\{pt\}{\rightrightarrows}\{pt\})\rightarrow \cG$, and $i: \cG \rightarrow \cG$, satisfying the usual conditions for the multiplication, identity element and inverses on a group.
\end{defi}
\begin{remark} Writing  $\Omega{\rightrightarrows} M$ for $\cG$,
in particular both the space of arrows $\Omega$ and the space of objects $M$ are groups.
\end{remark}

\begin{lemma} \cite{BL}\cite[Sec. 3]{FB}.
Crossed modules of Lie groups  are in one-to-one correspondence  to Lie 2-groups.
\end{lemma}

We recall one direction of this correspondence in Section \ref{sec:cga}.

\bibliographystyle{habbrv}
\bibliography{bibPSR}
\end{document}